\documentclass[11pt,reqno]{amsart}
 \usepackage{graphicx}
\usepackage{amscd,amsmath,amsopn,amssymb,amsthm,multicol}
\usepackage[color,matrix, all, 2cell]{xy}
\usepackage{amscd}
\usepackage{lscape}
\usepackage{slashed}
\usepackage{graphicx}
\usepackage{setspace}
\usepackage{upgreek}
\usepackage{textgreek}
\usepackage{enumerate}
\usepackage{color}
\usepackage{lscape}
\usepackage{tikz}
\usepackage{multirow}
\usepackage{cancel}
\usepackage{soul}
\usepackage{harmony}
\usepackage{comment}
\usepackage{wasysym}
\usepackage{mathrsfs}
\usepackage{mathtools}

\numberwithin{equation}{section}

\DeclareMathAlphabet{\mathscrbf}{OMS}{mdugm}{b}{n}

\DeclareMathOperator{\diag}{diag}
\DeclareMathOperator{\Ad}{Ad}
\DeclareMathOperator{\ad}{ad}

\DeclareMathOperator{\Id}{Id}
\DeclareMathOperator{\tr}{tr}

\DeclareMathOperator{\Span}{span}

\DeclareMathOperator{\rk}{rk}

\DeclareMathOperator{\dd}{d}

\newcommand{\fr}{\mathfrak}
\newcommand{\al}{\alpha}
\newcommand{\be}{\beta}

\newcommand{\bb}{\mathbb}
\newcommand{\mc}{\mathcal}

\DeclareMathOperator{\Spin}{\mathsf{Spin}}
\DeclareMathOperator{\SO}{\mathsf{SO}}
\DeclareMathOperator{\Sp}{\mathsf{Sp}}
 \DeclareMathOperator{\SU}{\mathsf{SU}}
  \DeclareMathOperator{\PSU}{\mathsf{PSU}}
  
\DeclareMathOperator{\U}{\mathsf{U}}
\DeclareMathOperator{\V}{\mathsf{V}}
\DeclareMathOperator{\Ee}{\mathsf{E}}
\DeclareMathOperator{\G}{\mathsf{G}}

\DeclareMathOperator{\E}{\mathsf{E}}
\DeclareMathOperator{\Ss}{\bb{S}}

\DeclareMathOperator{\Gl}{\mathsf{GL}}

\theoremstyle{plain}
\newtheorem{lemma}{Lemma} [section]

\newtheorem{corol}[lemma] {Corollary}
\newtheorem{prop} [lemma]{Proposition}
\newtheorem{thm}{Theorem}

\theoremstyle{definition}
\newtheorem{definition}[lemma] {Definition}
\newtheorem{example}[lemma] {Example}

\newtheorem{remark}[lemma] {Remark}
\newtheorem*{remark*}{Remark}

\definecolor{dark}{rgb}{0.18,0.18,0.68}
\definecolor{mydark}{rgb}{0.78,0.08,0.08}
\definecolor{crew}{rgb}{0.2,0.5,0.2}
\definecolor{mmg}{rgb}{0.31,0.50,0.23}
\definecolor{dblue}{rgb}{0.01,0.01,0.44}
\definecolor{red}{rgb}{0.57,0.11,0.15}
\definecolor{cobalt}{RGB}{61,89,171}
\usepackage[colorlinks,citecolor=cobalt,linkcolor=cobalt,urlcolor=cobalt,pdfpagemode=UseNone,backref = page]{hyperref}

 \input ulem.sty

\title[Homogeneous 8-manifolds admitting invariant $\mathrm{Spin}(7)$-structures]{Homogeneous 8-manifolds admitting invariant $\mathbf{Spin(7)}$-structures}

\author{Dmitri Alekseevsky}
\address{Institute for Information Transmission Problems, B. Karetny per. 19, 127051, Moscow, Russia  and
Faculty of Science, University of Hradec Kr\'alov\'e, Rokitanskeho 62, Hradec Kr\'alov\'e
50003, Czech Republic}

\email{dalekseevsky@iitp.ru}

\author{Ioannis Chrysikos}
\address{Faculty of Science, University of Hradec Kr\'alov\'e, Rokitanskeho 62, Hradec Kr\'alov\'e
50003, Czech Republic}
\email{ioannis.chrysikos@uhk.cz}

\author{Anna Fino}
\address{Dipartimento di Matematica ``G. Peano'', Universit\`a  degli Studi di Torino, Via Carlo Alberto 10, 10123 Torino, Italy}
\email{annamaria.fino@unito.it}

\author{Alberto Raffero}
\address{Dipartimento di Matematica ``G. Peano'', Universit\`a  degli Studi di Torino, Via Carlo Alberto 10, 10123 Torino, Italy}
\email{alberto.raffero@unito.it}

\begin{document}

\begin{abstract}
We study compact, simply connected, homogeneous 8-manifolds admitting invariant $\Spin(7)$-structures,
classifying all canonical presentations $G/H$ of such spaces, with $G$ simply connected.
For each presentation, we exhibit explicit examples of invariant $\Spin(7)$-structures and we describe their type, according to Fern\'andez classification.
Finally, we analyse the associated $\Spin(7)$-connection with torsion.
\end{abstract}

\maketitle


\section{Introduction}\label{intro}
A $\Spin(7)$-structure on an 8-manifold $M$ is characterized by the existence of a 4-form $\Phi$ which can be pointwise written as
\begin{eqnarray*}
\left.\Phi\right|_x	&=& e^{0123}+e^{0145}-e^{0167}+e^{0246}+e^{0257}+e^{0347}-e^{0356}\nonumber\\
				&  &+e^{4567}+e^{2367}-e^{2345}+e^{1357}+e^{1346}-e^{1247}+e^{1256}, \nonumber
\end{eqnarray*}
for some basis $(e^0,\ldots,e^7)$ of the cotangent space $T^*_xM$, where $e^{ijkl}$ denotes the wedge product of covectors $e^i\wedge e^j \wedge e^k \wedge e^l$.
Any such form is called {\it admissible}, and it gives rise to a Riemannian metric $g_\Phi$ and an orientation on $M$ by the inclusion
$\Spin(7)\subset\SO(8)$ (cf.~\cite{Joy96}). The existence of $\Spin(7)$-structures depends on the topology of the manifold \cite{Law}. In particular, $M$ has to be orientable and spin.

By \cite{Ferr}, the intrinsic torsion of a $\Spin(7)$-structure can be identified with the exterior derivative of the corresponding 4-form $\Phi$.
As a consequence, the Riemannian holonomy group $\mathrm{Hol}(g_\Phi)$ is a subgroup of $\Spin(7)$ if and only if $ {\rm d}\Phi=0$.
In such a case, the metric $g_\Phi$ is Ricci-flat, and the $\Spin(7)$-structure is said to be {\it torsion-free}.
More generally, the decomposition of the $\Spin(7)$-module $\Lambda^5((\bb{R}^8)^*)$ into irreducible submodules allows one to divide $\Spin(7)$-structures into four classes, which were first described in \cite{Ferr}.
Recently, a description of these classes in terms of spinorial equations has been obtained in \cite{Martin-Merchan}.

Since $\Phi$ is parallel with respect to the Levi Civita connection $\nabla^{g_{\Phi}}$ of $g_\Phi$ if and only if $\mathrm{Hol}(g_\Phi)\subseteq\Spin(7)$,
any other linear connection on $M$ preserving the $\Spin(7)$-structure $\Phi$ must necessarily have torsion.
By \cite{Iv}, on $(M,\Phi,g_{\Phi})$ there exists a unique connection with totally skew-symmetric torsion $T$ preserving both $\Phi$ and a nontrivial spinor.
It is given by $\nabla \coloneqq \nabla^{g_{\Phi}}  + \frac 12 T$.

The aim of this article is to study compact, simply connected 8-manifolds endowed with a $\Spin(7)$-structure and acted on (almost) effectively and transitively
by a compact connected Lie group $G$ of automorphisms.
Every such manifold admits a presentation of the form $(M=G/H,\Phi)$, where $H$ is a closed subgroup of $G$, and $\Phi$ is a $G$-invariant admissible 4-form.
From the algebraic point of view, $M=G/H$ is a compact, simply connected, (almost) effective homogeneous space whose isotropy action on the tangent space is equivalent to the action
of a closed subgroup of $\Spin(7)$ on $\bb{O}\cong\bb{R}^{8}$.

Recall that every compact, simply connected, homogeneous space $M$ admits a {\it canonical presentation}, that is, it can be written as $M=G/H$
with $G$ a compact, connected, simply connected, semisimple Lie group and $H$ a closed connected subgroup of $G$ (see e.g.~\cite{Bohm,Oni}).
All compact, simply connected, homogeneous 8-dimensional manifolds $G/H$ of a compact, connected, simply connected Lie group $G$  were classified in \cite{Klaus}.
Moreover, a classification of compact simply connected Riemannian symmetric spin manifolds was given in \cite{Cah}.
A topological examination of these spaces allows us to obtain the following.

\begin{thm}\label{MainTheoremA}\
\begin{enumerate}[a)]
\item\label{MainA} The canonical presentations of all compact, simply connected, non-symmetric almost effective homogeneous spaces admitting a $\Spin(7)$-structure are exhausted by \smallskip
\begin{enumerate}[1)]
\item $\displaystyle\frac{\SU(3)}{\{e\}}$;\smallskip
\item\label{CklmMain}  $C_{k, \ell, m}\coloneqq\displaystyle\frac{\SU(2)\times\SU(2)\times\SU(2)}{\U(1)_{k, \ell, m}}, \quad k\geq \ell\geq m \geq0,~k>0,~{\rm gcd}(k, \ell, m)=1$; \smallskip
\item\label{nKMain} $\displaystyle\frac{\SU(2)\times\SU(2)\times\SU(2)}{\Delta\SU(2)}\times\frac{\SU(2)}{\U(1)}$; \smallskip	
\item\label{CaEcMain}  $\displaystyle\frac{\SU(3)}{\SU(2)}\times\SU(2)$.
\end{enumerate}
As smooth manifolds, the spaces \ref{CklmMain}) and \ref{nKMain}) are diffeomorphic to $\Ss^{3}\times\Ss^{3}\times\Ss^{2}$,
while the space \ref{CaEcMain}) is diffeomorphic to $\Ss^{5}\times\Ss^{3}$.
\item\label{MainB}  The compact, simply connected, symmetric spaces admitting a $\Spin(7)$-structure are exhausted by $\SU(3)$, $\Ss^{3}\times\Ss^{3}\times\Ss^2$, $\Ss^{5}\times\Ss^{3}$,
$\bb{HP}^2$, ${\rm Gr}_{2}(\bb{C}^{4})$	 and the exceptional Wolf space $\displaystyle\frac{\G_2}{\SO(4)}$.
\end{enumerate}
\end{thm}

The manifold $C_{k, \ell, m}$ appearing in the above theorem is a torus bundle over the homogeneous K\"ahler-Einstein manifold $(\SU(2)/\U(1))^{\times3}$,
and hence a non-K\"ahler C-space. Invariant Einstein metrics on it were discussed in \cite{Bohm, WZ2}.
When $m=0$, $C_{k, \ell, 0}$ is the direct product of $\Ss^3$ with the total space of a circle bundle over $\Ss^2\times\Ss^2$.
Furthermore, $C_{1,0,0}=\Spin(4)\times\frac{\SU(2)}{\U(1)}$.

On the other hand, the homogeneous space $(\SU(3)/\SU(2))\times\SU(2)$ is an example of a Calabi-Eckmann manifold.
This is a torus bundle over $\bb{CP}^{2}\times\bb{CP}^{1}$, and hence also a non-K\"ahler C-space.
By Jensen \cite{Jen}, the 5-sphere $\SU(3)/\SU(2)$ admits a unique invariant Einstein metric which coincides with the canonical metric.
Consequently,  the space $(\SU(3)/\SU(2))\times\SU(2)$ admits a unique invariant Einstein metric.

Using general properties of symmetric spaces, we see that there are no invariant $\Spin(7)$-structures on the manifolds  described  in part  \ref{MainB}) of Theorem \ref{MainTheoremA}
(cf.~Lemma \ref{symm}).
Combining this with a case-by-case analysis of the homogeneous spaces appearing in part \ref{MainA}) gives the following result.
\begin{thm}\label{MainTheoremB}
The canonical presentations of all compact, simply connected, almost effective homogeneous spaces admitting an invariant $\Spin(7)$-structure are exhausted by
$\frac{\SU(3)}{\{e\}}$, the infinite family $C_{k, \ell, m}$, for $k=\ell+m$, and the Calabi-Eckmann manifold $\frac{\SU(3)}{\SU(2)}\times\SU(2)$.
\end{thm}

It is remarkable that  there are just a few examples of compact simply connected homogeneous spaces admitting invariant $\Spin(7)$-structures.
This is different from the case of $\G_2$-structures, where examples of this type are abundant (see \cite{Le,Rei},
and compare with the classification of compact almost effective homogeneous 7-manifolds given in \cite{ACT}).

An example of invariant $\Spin(7)$-structure inducing the bi-invariant metric on the homogeneous space $\SU(3)/\{e\}$ was described in \cite{Ferr}.
For the spaces $C_{k, \ell,m}$, with $k>\ell > m > 0$, and $(\SU(3)/\SU(2))\times\SU(2)$, we describe the invariant Riemannian metrics and the invariant differential forms.
This allows us to obtain a $5$-parameter family of invariant $\Spin(7)$-structures of mixed type on both of them.
In particular, for both spaces we show that there exists an invariant $\Spin(7)$-structure $\Phi$ inducing the normal metric and whose associated $\Spin(7)$-connection $\nabla$
coincides with the canonical connection corresponding to the naturally reductive structure induced by $g_{\Phi}$. From this, it follows that  $\nabla$  has parallel torsion.

The paper is organized as follows.
In Section \ref{prel}, we recall some basic facts on $\Spin(7)$-structures. Homogeneous 8-manifolds with an invariant $\Spin(7)$-structure are discussed in Section \ref{homspin7}.
In Section \ref{ss}, we review the main properties of simply connected homogeneous spaces.
The main theorems \ref{MainTheoremA} and \ref{MainTheoremB} are proved in Sections \ref{proofTHEMA} and \ref{proofB}, respectively.
In particular, the infinite family $C_{k, \ell, m}$ is studied in Section \ref{cklm},  the Calabi-Eckman manifold $(\SU(3)/\SU(2))\times\SU(2)$ is described in Section \ref{ss5ss3}
and $(\SU(2) \times \SU(2) \times \SU(2))/\Delta(\SU(2)))\times(\SU(2)/\U(1))$  is analysed in Section \ref{Space4}.
Explicit examples of invariant $\Spin(7)$-structures on $C_{\ell+m, \ell, m}$, with $\ell\geq m>0$, and on the Calabi-Eckmann manifold are given is Section \ref{genericexa},
where we also study the corresponding invariant $\Spin(7)$-connection with torsion.
Finally, in Appendix \ref{last} we discuss the classification of all non-symmetric homogeneous presentations of $\Ss^3\times\Ss^3\times\Ss^2$.

We emphasize that the results of this paper are also useful to study compact 8-manifolds admitting other types of special structures, e.g.~invariant $\PSU(3)$-structures.
This will be discussed in a forthcoming work.


\section{Preliminaries}\label{prel}

We begin recalling the main properties of 8-manifolds whose frame bundle admits a reduction to the Lie group $\Spin(7)\subset \SO(8)$.
For more details, we refer the reader to \cite{Br1,Iv, Joy96, Law}.

Consider the vector space $\bb{R}^8$, denote by $\{e_0,\ldots,e_7\}$ the canonical basis, and by $\{e^0,\ldots,e^7\}$ its dual basis.
The group $\Spin(7)$ can be defined as the stabilizer in $\Gl(8,\bb{R})$ of the following 4-form on $\bb{R}^8$:	
\begin{eqnarray}
\Phi_{\scriptscriptstyle0}	&=&	e^{0123}+e^{0145}-e^{0167}+e^{0246}+e^{0257}+e^{0347}-e^{0356}\nonumber\\
					&  &+e^{4567}+e^{2367}-e^{2345}+e^{1357}+e^{1346}-e^{1247}+e^{1256}. \label{4f}
\end{eqnarray}
$\Spin(7)$ is a compact, connected, simply connected Lie group of dimension 21.
It is a subgroup of $\SO(8)$, as it preserves both the Euclidean inner product $g_{\scriptscriptstyle0} = \sum_{i=1}^8(e^i)^2$ on $\bb{R}^8$ and the volume form
$\frac{1}{14}\Phi_{\scriptscriptstyle0}\wedge\Phi_{\scriptscriptstyle0} = e^{01234567}$. Moreover, its center is $\bb{Z}_2=\left\{\pm \mathrm{Id}_{\bb{R}^8} \right\}$ and  $\Spin(7)/\bb{Z}_2\cong \SO(7)$
(see \cite[Thm.~4]{Br1} for a proof).

A $\Spin(7)$-structure on a $8$-dimensional manifold $M$ is a reduction of the structure group of its frame bundle from $\Gl(8,\bb{R})$ to $\Spin(7)$.
As $\Spin(7)$ is the stabilizer of the 4-form $\Phi_{\scriptscriptstyle0}$, such a reduction is characterized by the existence of a globally defined 4-form $\Phi\in\Omega^4(M)$
which can be pointwise identified with $\Phi_{\scriptscriptstyle0}$ by means of an isomorphism $u:T_xM\rightarrow \bb{R}^8$.
Any such form is called {\it admissible}, and it gives rise to a Riemannian metric $g_\Phi$ and to an orientation $dV_\Phi$ on $M$ by the inclusion $\Spin(7)\subset \SO(8)$.
We denote the Hodge operator associated with this metric and orientation by $\star$. Notice that $\Phi$ is self-dual, i.e., $\star\Phi=\Phi$.
An explicit description of the metric $g_\Phi$ in terms of the 4-form $\Phi$ can be found, for instance, in \cite[Sect.~4.3]{Kar}.

\begin{remark}
By dimension counting, the $\Gl(8,\bb{R})$-orbit of $\Phi_{\scriptscriptstyle0}$ is not open in $\Lambda^4((\bb{R}^8)^*)$.
Consequently, an admissible 4-form $\Phi$ is not {\it stable} in the sense of Hitchin \cite{Hit}.
This differs significantly from the case of $\G_2$-structures on 7-manifolds, which are defined by stable 3-forms satisfying a suitable positivity condition.
In eight dimensions, stability occurs for 3- and 5-forms, and the corresponding geometric structures are related to the group $\PSU(3)$.
\end{remark}

Since $\Spin(7)$ is both connected and simply connected, a connected 8-manifold $M$ admitting a $\Spin(7)$-structure must be orientable and spin (with a preferred spin structure and orientation).
These conditions are equivalent to the vanishing of the first and second Stiefel-Whitney classes of $M.$
However, not every 8-dimensional Riemannian spin manifold admits $\Spin(7)$-structures. More precisely, this is a topological issue, which can be characterized as follows.
\begin{prop}[\cite{Law}]\label{top}
An 8-dimensional orientable spin manifold $M$ admits $\Spin(7)$-structures if and only if, for an appropriate choice of orientation,
the following equation involving the Pontryagin classes $p_{1}(M)$, $p_{2}(M)$ and the Euler characteristic $\chi(M)$ of $M$ holds
\[
p_{1}^{2}(M)-4p_{2}(M)+8\chi(M) = 0
\]
\end{prop}

The intrinsic torsion of a $\Spin(7)$-structure can be identified with the covariant derivative of the defining 4-form $\Phi$ with respect to the Levi Civita connection $\nabla^{g_\Phi}$ of $g_\Phi$.
When $\nabla^{g_\Phi}\Phi=0$, the intrinsic torsion vanishes identically, the holonomy group of $g_\Phi$ is a subgroup of $\Spin(7)$ and $g_\Phi$ is Ricci-flat.
In such a case, the $\Spin(7)$-structure is said to be {\it torsion-free} or {\it parallel}.
By \cite{Ferr}, the intrinsic torsion $\nabla^{g_\Phi}\Phi$ can be also identified with the 5-form $ {\rm d}\Phi$.
Moreover, the $\Spin(7)$-module $\Lambda^5((\bb{R}^8)^*)$ splits into the direct sum of two irreducible submodules,
say $\Lambda^5((\bb{R}^8)^*) \cong \mc{W}_1 \oplus \mc{W}_2$,
with $\dim(\mc{W}_1) = 48$, $\dim(\mc{W}_2) = 8$.
This allows one to divide $\Spin(7)$-structures into four classes, which are completely characterized by ${\rm d}\Phi$.
Besides the class of parallel $\Spin(7)$-structures, corresponding to the condition ${\rm d}\Phi=0$, the following possibilities occur
\begin{enumerate}[$\bullet$]
\item class $\mc{W}_1$: {\it balanced} $\Spin(7)$-structures, characterized by the condition $\star{\rm d}\Phi\wedge \Phi=0$;
\item class $\mc{W}_2$: {\it locally conformal parallel} $\Spin(7)$-structures, characterized by the condition $ {\rm d}\Phi = \vartheta\wedge\Phi$;
\item class $\mc{W}_1\oplus \mc{W}_2$: $\Spin(7)$-structures of {\it mixed type}.
\end{enumerate}
The 1-form $\vartheta$ is given by
\[
\vartheta = -\frac{1}{7}\star(\star{\rm d}\Phi\wedge \Phi)=\frac{1}{7}\star(\delta\Phi\wedge\Phi),
\]
and it is called the {\it Lee form} of the $\Spin(7)$-structure.
In particular, $\Phi$ is balanced if and only if $\vartheta=0$, while ${\rm d}\vartheta=0$ whenever the $\Spin(7)$-structure is locally conformal parallel (l.c.p.~for short), see e.g.~\cite[Lemma 4.5.2]{Kar}.

Finally, according to \cite[Thm.~1.1]{Iv}, any 8-dimensional manifold $M$ endowed with a $\Spin(7)$-structure $\Phi$ admits a unique metric connection  $\nabla$ with totally skew-symmetric torsion $T$,
satisfying $\nabla\Phi=0$.
It is given by $\nabla \coloneqq \nabla^{g_\Phi}+\frac{1}{2}T$, where $T=-\delta{\Phi}-\frac{7}{6}\star(\vartheta\wedge{\Phi})$, and it is called the {\it  characteristic connection} of $(M,\Phi)$.


\section{Invariant $\Spin(7)$-structures on homogeneous spaces}\label{homspin7}
We now focus on homogeneous spaces. We start with the following definition.
\begin{definition}\label{HomSpin}
A $\Spin(7)$-structure $\Phi$ on an  8-dimensional manifold $M$ is called {\it homogeneous} or {\it invariant} if there exists a connected Lie group $G$ acting transitively and almost effectively on $M$,
preserving the 4-form $\Phi$.
\end{definition}

In this case, $M$ is $G$-equivariantly diffeomorphic to the homogeneous space $G/H$, where $H$ is the  (compact) stability group of a fixed point $o\in M,$
and $\Phi$ is a $G$-invariant 4-form on $G/H$ with pointwise stabilizer isomorphic to $\Spin(7)$.
Equivalently, the isotropy subgroup $\chi(H)\subset \Gl(T_oM)$ is a subgroup of $\Spin(7)$, where $\chi:H\rightarrow\Gl(T_oM)$ denotes the isotropy representation of $M=G/H$.
Conversely, a homogeneous 8-manifold $M=G/H$ with $\chi(H)\subseteq\Spin(7)$ admits invariant $\Spin(7)$-structures.

As we are interested in compact examples, from now on we assume that $G$ is compact.
Then $H$ is compact as well, and the Lie algebra $\mathfrak{g}$ of $G$ admits a reductive decomposition $\mathfrak{g}=\mathfrak{h}\oplus\mathfrak{m}$,
where $\mathfrak{h}$ is the Lie algebra of $H$ and $\mathfrak{m}$ is an $\mathrm{Ad}(H)$-invariant subspace of $\mathfrak{g}$.
Moreover, we can identify $\mathfrak{m}$  with the tangent space $T_{o}M$ and
the $G$-invariant 4-form $\Phi$ on $G/H$ with an $\mathrm{Ad}(H)$-invariant 4-form on $\mathfrak{m}$, which we shall denote by the same letter.
Since the $G$-action on $M=G/H$ is almost effective,
the isotropy representation $\chi_* : \fr{h} \to \fr{gl}(\mathfrak{m})$ is injective, and we can identify the subalgebra $\fr{h}$ with the isotropy subalgebra $\chi_{*}(\fr{h})$ of
the Lie algebra $\fr{spin}(7)\subset \fr{gl}(\mathfrak{m})$.
Notice that the following constraints must hold
\[
\mathrm{dim}(\mathfrak{g}) = \mathrm{dim}(\mathfrak{h}) + 8,\quad   \rk\fr{h}\leq 3.
\]

The well-known interplay between $\G_2$- and $\Spin(7)$-structures (see e.g.~\cite{Cab95})
implies that every invariant $\G_2$-structure on a compact homogeneous 7-manifold $N=L/K$ gives rise to an invariant $\Spin(7)$-structure
on $M=L/K\times\U(1)$ and, conversely, every invariant $\Spin(7)$-structure on the 8-manifold $M=L/K\times\U(1)$ induces an invariant $\G_2$-structure on $N=L/K$.
Consequently, in this case the complete list of homogeneous manifolds with an invariant $\Spin(7)$-structure can be obtained from the results of \cite{Le,Rei}.

In the next sections, we will deal with the classification of the canonical presentations of compact, simply connected, almost effective homogeneous 8-manifolds
that admit invariant $\Spin(7)$-structures.
The strategy to study this problem is the following.
First, we consider all  compact, simply connected, almost effective, homogeneous 8-manifolds with their {\it canonical presentation}, and we determine those satisfying the characterization
of Proposition \ref{top}. This gives the list of the spaces admitting $\Spin(7)$-structures.
Then, for each space we investigate whether there exists an invariant admissible 4-form.

We conclude this section with some remarks.
\begin{lemma}\label{symm}
A compact, simply connected, Riemannian symmetric space cannot admit any invariant $\Spin(7)$-structure.
\end{lemma}
\begin{proof}
Since all invariant differential forms on a Riemannian symmetric space are closed (see e.g.~\cite[p.~250]{Wolf}), any invariant $\Spin(7)$-structure on a Riemannian symmetric space must be
torsion-free. In particular, the corresponding  invariant metric must be Ricci-flat.
However, every compact simply connected Riemannian symmetric space is a direct product of irreducible symmetric spaces of compact type, which are
Einstein with non-zero Einstein constant (cf.~\cite[10.83]{Bes}).
\end{proof}

More generally, since Ricci-flat homogeneous manifolds are flat \cite{AK75}, the class of compact connected homogeneous spaces admitting an invariant torsion-free $\Spin(7)$-structure
is exhausted by flat tori.


\section{Simply connected 8-dimensional homogeneous spaces}\label{ss}
Let $M$ be a compact, simply connected homogeneous space and let $G'$ be a connected Lie group acting transitively and almost effectively on it.
Starting from the corresponding presentation $M=G'/H'$ and using the results of \cite{Mon,Oni}, it is always possible to obtain a presentation of the form $M=G/H$,
where $G$ is a compact, connected, simply connected, semisimple Lie group and $H\subset G$ is a connected closed subgroup (see e.g.~\cite{Bohm1,Bohm} for more details).
This motivates the following.
\begin{definition}
Let $M$ be a compact, simply connected homogeneous space. A {\it canonical presentation} of $M$ is a presentation of the form $M=G/H$,
where $G$ is a compact, connected, simply connected, semisimple Lie group and $H\subset G$ is a connected closed subgroup.
\end{definition}
In what follows, we restrict our attention to such presentations.

\begin{prop}\label{bom} \textnormal{(\cite{Bohm1})}
For a compact, simply connected, homogeneous space $M$ with canonical presentation $M=G/H$ with $H$ semisimple, two possible cases occur:
\begin{enumerate}[(I)]
\item\label{BohmI}  $\rk N_{G}(H) = \rk H$;
\item\label{BohmII} $\rk N_{G}(H) >\rk H$.
\end{enumerate}
In the first case, $M=G/H$ is the direct product of indecomposable homogeneous spaces $G_i/H_i$ which also satisfy {\rm (\ref{BohmI})}, that is
\[
M=G/H=G_1/H_1\times\ldots\times G_k/H_k,
\]
for some $k\geq 1$, with $G_i$ compact connected, simply connected and semisimple and $H_{i}\subset G_{i}$ closed, for all $1\leq i\leq k$.
Such homogeneous spaces $G_i/H_i$ are called {\it prime}.

In  case {\rm (\ref{BohmII})},  $M=G/H$ is the total space of a principal torus bundle over a product of prime homogeneous spaces.
In particular, for any maximal torus ${\rm T}$ in a compact complement of $H$ in $N_{G}(H)$, $M=G/H $ is the total space of the principal torus bundle
\[
(H\cdot {\rm T})/H\to G/H \to G/(H\cdot {\rm T}),
\]
where  $H\cdot {\rm T}$ denotes  $(H\times {\rm T})/H\cap {\rm T}$.
Note that the base space $G/(H\cdot {\rm T})$ does not depend on the choice of ${\rm T}$ and $\rk N_{G}(H\cdot {\rm T})=\rk(H\cdot {\rm T})$.
\end{prop}

\begin{example}
Consider  the simply connected coset $(\G_2\times\SU(2))/(\SU(3)\times\U(1))$, where $\SU(3)$ is the maximal compact connected subgroup of $\G_2$, and $\U(1)\subset \SU(2)$ is a
maximal torus. It is easy to see that this space satisfies condition (\ref{BohmI}).
Hence, it is the direct product of the prime homogeneous spaces $\bb{S}_{\mathrm{irr}}^6=\G_2/\SU(3)$ and $\bb{CP}^1=\SU(2)/\U(1)$.
\end{example}

Proposition \ref{bom} allows us to distinguish two classes of compact, simply connected homogeneous spaces.
A particular example of (\ref{BohmII}) is the following (cf.~\cite[p.~80]{Klaus}).

\begin{lemma}[\cite{Klaus}]
Let $P$ be the total space of a ${\rm T}^{q-1}$-principal bundle over $(\Ss^{2})^{\times q}$, with $q\geq 2$.
If $P$ is simply connected, then it is diffeomorphic to $\Ss^{2}\times(\Ss^{3})^{\times (q-1)}$.
\end{lemma}

Notice that for $q=2$ one obtains the circle bundle
\[
\Ss^{1}\to M_{k, \ell}=(\SU(2)\times\SU(2))/\U(1)_{k, \ell}\to\Ss^{2}\times\Ss^{2}.
\]
The space $M_{k, \ell}$ is a compact simply connected 5-dimensional spin manifold  with $H^{2}(M_{k, \ell}, \bb{Z})=\bb{Z}$, and hence diffeomorphic to the product of spheres $\Ss^{2}\times\Ss^{3}$.
For $k=\ell=1$,  the space $M_{1, 1}$ is diffeomorphic to $\SO(4)/\SO(2)$, and it can be viewed as the unit tangent bundle of $\Ss^{3}$, see also \cite[p.~6358]{Nik2}.
For $q=3$, we get  the  space
\[
C_{k, \ell , m} \coloneqq  (\SU(2)\times\SU(2)\times\SU(2))/\U(1)_{k, \ell ,m},
\]
where the embedding $\U(1)_{k, \ell, m} \subset \SU(2)\times\SU(2)\times\SU(2)$ will be specified later (see Section \ref{cklm}).
When ${\rm gcd}(k, \ell, m)=1$, the space $C_{k, \ell, m}$ is  a torus bundle over $\Ss^{2}\times\Ss^{2}\times\Ss^{2}$
\[
{\rm  T}^{2}\to (\SU(2)\times\SU(2)\times\SU(2))/\U(1)_{k, \ell ,m}\to \Ss^{2}\times\Ss^{2}\times\Ss^{2}\,,
\]
see Lemma \ref{my11} for a proof.
Thus, as a manifold it is diffeomorphic to $ \Ss^{3}\times\Ss^{3}\times\Ss^{2}$.

An inspection of the list of canonical presentations given in \cite{Bohm} (see also \cite[Table 1]{ArLa}) combined with the results of \cite[p.~81]{Klaus}, allows us to obtain
the canonical presentations for all compact, connected, simply connected, spin non-symmetric almost effective homogeneous 8-manifolds. They are described in Table \ref{Table1}.

More details on the cosets $(2)-(4)$ are given in Section \ref{proofB}, while further presentations of $\Ss^3\times\Ss^3\times\Ss^2$ are discussed in Appendix \ref{last}.
Notice that the spaces $(3)$ and $(6)-(8)$ are all simply-connected homogeneous nearly K\"ahler manifolds.

\begin{table}[ht]
\centering
\renewcommand\arraystretch{2.2}
\begin{tabular}{| c | l | l |}
\hline
	&	$M^8$								& 	canonical presentation $G/H$									 	\\ \hline
(1)	&	$\SU(3)$								&	$\displaystyle\frac{\SU(3)}{\{e\}}$													\\
(2)	&	$\Ss^{3}\times\Ss^{3}\times\Ss^{2}$			&	$C_{k, \ell, m}\coloneqq\displaystyle\frac{\SU(2)\times\SU(2)\times\SU(2)}{\U(1)_{k, \ell, m}},\quad k\geq\ell\geq m\geq0,~k>0,~{\rm gcd}(k, \ell, m)=1$	\\
(3)	&	$\Ss^{3}\times\Ss^{3}\times\Ss^{2}$			&	$\displaystyle\frac{\SU(2)\times\SU(2)\times\SU(2)}{\Delta\SU(2)}\times\frac{\SU(2)}{\U(1)}$										 \\
(4)	&	$\Ss^{5}_{\V\oplus\bb{R}}\times\Ss^{3}$					&	$\displaystyle\frac{\SU(3)}{\SU(2)}\times\SU(2)$ \\
(5)	&	$\Sp(2)$-full flag  						& 	$\displaystyle\frac{\Sp(2)}{{\rm T}^{2}_{\rm max}}$ \\
(6)	&	$\bb{F}^3\times\Ss^{2}$ 		&	$\displaystyle\frac{\SU(3)}{{\rm T}^{2}_{\rm max}}\times\displaystyle\frac{\SU(2)}{\U(1)}$\\
(7)	&	$\bb{CP}^{3}_{\fr{m}_1\oplus\fr{m}_2}\times\Ss^{2}$ 				&	$\displaystyle\frac{\Sp(2)}{\Sp(1)\times\U(1)}\times\displaystyle\frac{\SU(2)}{\U(1)}$ \\
(8)	&	$\Ss^{6}_{\rm irr}\times\Ss^{2}$ 					& $\displaystyle\frac{\G_2}{\SU(3)}\times\displaystyle\frac{\SU(2)}{\U(1)}$ \\   \hline
\end{tabular}
\vspace{0.1cm}
\caption{Canonical presentations of compact simply connected    spin almost effective non-symmetric homogeneous 8-manifolds.}\label{Table1}
\end{table}


\section{Proof of Theorem \ref{MainTheoremA}}\label{proofTHEMA}
In this section, we study the existence of $\Spin(7)$-structures on the spaces appearing in Table \ref{Table1} using the topological characterization of Proposition \ref{top}.
Since all of these manifolds are orientable and spin,  we only need to examine the constraint
\begin{equation}\label{toptest}
8\chi(M) = 4p_{2}(M)-p_{1}^{2}(M).
\end{equation}
 Recall that for a compact, connected, oriented 8-manifold $M$, the following identity holds (cf.~\cite{Sal})
\begin{equation}\label{pot1}
\sigma(M)=\frac{1}{45}\langle 7p_2(M)-p_{1}^{2}(M), [M]\rangle,
\end{equation}
where $\sigma(M)$ is the {\it signature} of $M,$ namely the signature of the quadratic form associated to
\[
Q : H^{4}(M, \bb{R})\times H^{4}(M, \bb{R})\to \bb{R}, \quad (\al, \be)\mapsto \langle \al\cup\be, [M]\rangle \coloneqq \int_{M}\al\wedge\be,
\]
and $[M]\in H_{8}(M, \bb{Z})$ is the fundamental homology class defined by the orientation.
Moreover, the {\it $\hat{A}$-genus} of $M$ is given by
\begin{equation}\label{pot2}
\hat{A}(M)=\frac{1}{5760}(7p_{1}^{2}(M)-4p_{2}(M)).
\end{equation}

Assume now that $M^8$ is also spin and let us denote by $\Sigma$ its spinor bundle and by $D^{g} : \Gamma(\Sigma)\to\Gamma(\Sigma)$ the Dirac operator associated to a Riemannian metric $g$ on it.
Then, $\Sigma$ decomposes as $\Sigma=\Sigma^{+}\oplus\Sigma^{-}$ and one can consider the index of the (half) Dirac operator $D^{g}_{+} : \Gamma(\Sigma^{+})\to\Gamma(\Sigma^{-})$,
which is given by  ${\rm ind}({D}_{+}^{g}) \coloneqq \dim\ker({D}_{+}^{g}) - \dim{\rm coker}({D}_{+}^{g})$.
By the Atiyah-Singer Index Theorem,  ${\rm ind}({D}_{+}^{g})$  coincides with the $\hat{A}$-genus, i.e., ${\rm ind}({D}_{+}^{g})=\langle \hat{A}(M), [M]\rangle$.
Moreover,  if $M$ admits a metric of positive scalar curvature, then $\hat{A}(M)=0$.

Let us now prove part \ref{MainA}) of Theorem \ref{MainTheoremA}.
\begin{prop}\label{topnew}
Among the manifolds described in Table \ref{Table1}, only those appearing in the first four rows admit $\Spin(7)$-structures.
\end{prop}

\begin{proof}
By \cite[Sect.~7]{Ferr}, we know that $\SU(3)$ admits a $\Spin(7)$-structure inducing the bi-invariant metric. Let $M$ be one of the manifolds $(2)-(4)$ of Table \ref{Table1}.
As $M$ is a product of spheres with at least one of odd-dimension,  it is parallelizable by \cite{Kerv} (see also \cite{Parton}).
Consequently, it admits $\Spin(7)$-structures. Explicit examples of admissible 4-forms can be easily expressed in terms of a global coframe $\{e^0,\ldots,e^7\}$ providing the absolute parallelism.

We now prove that the remaining spaces of Table \ref{Table1} do not satisfy the relation \eqref{toptest}.
Indeed, apart from the full flag manifold $\Sp(2)/{\rm T}^2_{\rm max}$, they are all products of the form $M=X^{6}\times\Ss^{2}$, where $X^{6}$ is a 6-dimensional compact homogeneous nearly K\"ahler manifold.
Therefore, it is easy to see that they satisfy $p_1^{2}(M)=p_2(M)=0$, but their Euler characteristic is non-zero, since they are quotients of Lie groups of the same rank.
For the same reason, the full flag manifold $M=\Sp(2)/{\rm T}^2_{\rm max}$ has $\chi(M)\neq 0$.
Moreover, $\sigma(M)=\hat{A}(M)=0$, and by \eqref{pot1} and \eqref{pot2} we deduce that $p_1^{2}(M)=p_2(M)=0$.
Thus, none of these manifolds satisfies \eqref{toptest}.
\end{proof}

We now prove part \ref{MainB}) of Theorem \ref{MainTheoremA}.
\begin{prop}
The 8-dimensional compact, simply connected, symmetric spaces admitting $\Spin(7)$-structures are exhausted by  the Lie group $\SU(3)$,
the product of spheres $\Ss^{3}\times\Ss^3\times\Ss^{2}$ and $\Ss^{5}\times\Ss^{3}$, the quaternionic projective space $\bb{HP}^{2}$, the Grassmannian   ${\rm Gr}_{2}(\bb{C}^{4})$  and the exceptional Wolf space $\frac{\G_2}{\SO(4)}$.
\end{prop}
\begin{proof}
Since an 8-manifold admitting $\Spin(7)$-structures is spin, we can focus on the list of compact simply connected spin symmetric spaces \cite{Cah}.
Up to a finite cover, we have to consider the following spaces
\begin{eqnarray*}
\SU(3)									&=&	(\SU(3)\times\SU(3))/{\Delta\SU(3)},	\\
\Ss^{2}\times\Ss^{2}\times\Ss^{2}\times\Ss^{2}		&=&	\left({\SU(2)}/{\U(1)}\right)^{\times4},	\\
\Ss^{3}\times\Ss^{3}\times\Ss^{2}				&=&	{\SO(4)}/{\SO(3)} \times  {\SO(4)}/{\SO(3)}\times {\SU(2)}/{\U(1)}, \\
\Ss^{4}\times\Ss^{2}\times\Ss^{2}				&=&	{\SO(5)}/{\SO(4)} \times {\SU(2)}/{\U(1)} \times {\SU(2)}/{\U(1)},  \\
\Ss^{4}\times\Ss^{4}							&=&	{\SO(5)}/{\SO(4)}\times {\SO(5)}/{\SO(4)},	\\
\Ss^{5}_{\rm sym}\times\Ss^{3}					&=&	\SO(6)/\SO(5) \times \SU(2) = \SO(6)/\SO(5) \times \SO(4)/\SO(3)\\
										&=&	\SU(4)/\Sp(2) \times \SU(2) = \SU(4)/\Sp(2) \times \SO(4)/\SO(3),\\
\Ss^{6}_{\rm sym}\times\Ss^{2}					&=&	\SO(7)/\SO(6) \times \SU(2)/\U(1),	\\
\Ss^{8}_{\rm sym}							&=&	\SO(9)/\SO(8),	\\
\bb{CP}^{3}_{\rm sym}\times\Ss^{2}				&=&	\SU(4)/\mathsf{S}(\U(3)\U(1)) \times \SU(2)/\U(1) = \SO(6)/\U(3) \times \SU(2)/\U(1), \\
{\rm Gr}_{2}(\bb{C}^{4})						&=&	{\rm Gr}_{2}^{+}(\bb{R}^{6}),	\\
\bb{HP}^{2}								&=&\Sp(3)/(\Sp(2)\times\Sp(1)),  \\
\bb{W}^8						                                  &=&\frac{\G_2}{\SO(4)}. 	
\end{eqnarray*}
\noindent Among these spaces, the half of them satisfy \eqref{toptest}, namely  $\SU(3)$, $\Ss^{3}\times\Ss^{3}\times\Ss^{2}$, $\Ss^{5}\times\Ss^{3}$, $\bb{HP}^{2}$, ${\rm Gr}_{2}(\bb{C}^{4})$ and $\bb{W}^8$.
For the first three of them, we have $\chi=\sigma=\hat{A}=0$.  As for the quaternionic projective space and the exceptional Wolf space, they both satisfy $\chi=3$, $\sigma=1$ and $\hat{A}=0$. Finally, the Grassmannian   ${\rm Gr}_{2}(\bb{C}^{4})$  is such that $\chi=6$, $\sigma=2$ and $\hat{A}=0$, hence $p_1^2=8$ and $p_2=14$.
\end{proof}


\section{Proof of Theorem \ref{MainTheoremB}}\label{proofB}
By Lemma \ref{symm}, we know that all compact, simply connected, symmetric spaces cannot admit invariant $\Spin(7)$-structures.
Moreover, an explicit example of an invariant $\Spin(7)$-structure inducing the bi-invariant metric on the homogeneous space $\SU(3)/\{e\}$ is constructed in \cite[Sect.~7]{Ferr}.
Thus, we only need to consider the remaining spaces appearing in part \ref{MainA}) of Theorem \ref{MainTheoremA}, namely
\[
C_{k, \ell, m}=\frac{\SU(2)\times\SU(2)\times\SU(2)}{\U(1)_{k, \ell, m}} ,  \quad
\frac{\SU(3)}{\SU(2)}  \times \SU(2), \quad \frac{\SU(2)\times\SU(2)\times\SU(2)}{\Delta\SU(2)}\times\frac{\SU(2)}{\U(1)}.
\]
In order to simplify the presentation, we examine each case separately.

\subsection{The infinite family $C_{k, \ell, m}$}\label{cklm}
Let $G \coloneqq \SU(2)\times\SU(2)\times\SU(2)$ and
\[
H \coloneqq \U(1)_{k, \ell, m} = \left\{(z^{k}, z^{\ell}, z^{m}) : z\in\U(1)\right\}.
\]
Denote the Lie algebra of  $G$ by $\fr{g} \coloneqq 3\fr{su}(2)=\fr{su}(2)\oplus\fr{su}(2)\oplus\fr{su}(2)$ and let $\fr{t}$ be the Lie algebra of a maximal torus of $G.$
The elements of $\fr{g}$ can be viewed as  $(6\times6)$ complex block matrices of the form ${\rm diag}(X, Y, Z)$, with $X, Y, Z\in\fr{su}(2)$.
Up to conjugation, any 1-dimensional subalgebra inside  $\fr{g}$ is described by a  homomorphism
\[
\rho_{k, \ell, m} : \fr{u}(1)\to \fr{g}, \quad  ix \mapsto \diag
\left(\begin{pmatrix}
ikx & 0 \\
0 & -ikx
\end{pmatrix},
\begin{pmatrix}
i\ell x & 0 \\
0 & -i\ell x
\end{pmatrix},
\begin{pmatrix}
imx & 0 \\
0 & -imx
\end{pmatrix}\right), \quad k, \ell, m\in\bb{R}.
\]
The image of $\rho_{k, \ell, m}$ is the Lie algebra of a closed connected subgroup of  $G$ if and only if $k, \ell, m\in\bb{Q}$.
Moreover, using the Weyl group and the outer automorphisms of $G$, it is always possible to reorder the elements of the triple $(k, \ell, m)$ in such a way that $k\geq \ell\geq m\geq 0$ and $k>0$
and assume that all $k, \ell, m$ are integers.
The stability algebra $\fr{h}$ is $\rho_{k, \ell, m}(\fr{u}(1))\coloneqq\fr{u}(1)_{k, \ell, m}$ and the Lie algebra $\fr{t}$ is given by
\[
\fr{t}=\left\{ {\rm diag}\left(
\begin{pmatrix}
ix & 0 \\
0 & -ix
 \end{pmatrix},
 \begin{pmatrix}
iy & 0 \\
0 & -iy
 \end{pmatrix},
\begin{pmatrix}
iz & 0 \\
0 & -iz
 \end{pmatrix} \right) : x, y, z\in\bb{R}\right\}\,.
\]
Since the stability group $H$ can be mapped by conjugation inside a maximal torus of $G$, any coset space of the form $(\SU(2))^{\times 3}/\U(1)$ is $G$-equivariantly diffeomorphic to  $C_{k, \ell, m}$,
where  $k, \ell, m$ are integers such that $k\geq \ell\geq m\geq 0$ and $k>0$.
Moreover, since we are interested in the simply connected case, we can assume that the  triple $(k, \ell, m)$ consists of  co-prime integers.

Denote by $\langle  \cdot , \cdot \rangle$ the bi-invariant metric on $G$ defined as $\langle A, B\rangle=-2\,{\rm tr}(AB)$
and let $\fr{m}=\fr{h}^{\perp}$ be the orthogonal complement of $\fr{h}$ inside $\fr{g}$ with respect to $\langle  \cdot , \cdot \rangle$.
Then, $[\fr{h}, \fr{m}]\subset\fr{m}$ and  $\fr{g}=\fr{h}\oplus\fr{m}$ is a reductive decomposition.
Thus, we can identify $\fr{m}$  with the tangent space $T_{o}C_{k, \ell, m}$ of $C_{k, \ell, m}$ at the identity coset $o\coloneqq eH$.
Let us consider the following orthogonal basis of $(\fr{g},\langle  \cdot , \cdot \rangle)$:
\begin{equation}\label{orth1}
\renewcommand\arraystretch{1.2}
\begin{array}{ll}
e_1 \coloneqq \diag(\sigma_1, 0, 0),  		&	e_2\coloneqq\diag(\sigma_2, 0, 0), \\
e_3\coloneqq\diag(0, \sigma_1, 0), 			&	e_4\coloneqq\diag(0, \sigma_2, 0), \\
e_5\coloneqq\diag(0, 0, \sigma_1),  			&	e_6\coloneqq\diag(0, 0, \sigma_2), \\
e_7\coloneqq\frac{1}{c_7}\diag(mk\sigma_3, m\ell\sigma_3, -(k^2+\ell^2)\sigma_3),  &	e_8\coloneqq\frac{1}{c_8}\diag(\ell\sigma_3, -k\sigma_3, 0), \\
e_9\coloneqq\diag(k\sigma_3, \ell\sigma_3, m\sigma_3), &  \\
\end{array}
\end{equation}
where
\[
\sigma_1\coloneqq \frac{1}{2}
\left ( \begin{array}{cc}
0 & i \\
i &  0
\end{array} \right ), \quad
\sigma_2 \coloneqq \frac{1}{2}
\left ( \begin{array}{cc}
0 & 1 \\
-1 &  0
\end{array} \right),  \quad
\sigma_3 \coloneqq \frac{1}{2}
\left ( \begin{array}{cc}
i & 0 \\
0 &  -i
\end{array} \right),
\]
$c_7\coloneqq\sqrt{k^{2}+\ell^{2}}\sqrt{k^{2}+\ell^{2}+m^{2}}$, and $c_8\coloneqq\sqrt{k^{2}+\ell^{2}}$.
Notice that $\langle e_i, e_i\rangle =1$, for all $i=1, \ldots, 8,$ and that $\langle e_9, e_9\rangle=k^{2}+\ell^{2}+m^{2}$.
Moreover,  $\fr{h}=\Span_{\bb{R}}\{e_9\}$ and $\fr{m}=\Span_{\bb{R}}\{e_1, \ldots,  e_8\}\cong\bb{R}^{8}$.

\begin{prop}\label{mi1}
The space $C_{k, \ell, m}=G/H$, with $k\geq \ell\geq  m\geq 0$, $k>0$ and  ${\rm gcd}(k, \ell, m)=1$, admits $G$-invariant $\Spin(7)$-structures  if and only if  $k-\ell-m=0$.
\end{prop}

\begin{proof}
Let us denote by $\{E_{ij} : 1\leq i<j\leq 8\}$ the basis of $\fr{so}(8)$ given by the skew-symmetric matrices $E_{ij}$ with $-1$ in the $(i, j)$-entry, $1$ in the $(j, i)$-entry and zeroes elsewhere.
The orthogonal transformation $\chi_{*}(e_9)|_{\fr{m}}\in\fr{so}(\fr{m})$ is given by
$\chi_{*}(e_9)|_{\fr{m}}=-kE_{12}-\ell E_{34}-mE_{56}$, since
\begin{equation}\label{adjon1}
\begin{array}{llll}
\ad(e_9)e_1=-ke_2, 	&	\ad(e_9)e_2=ke_1,	&	\ad(e_9)e_3=-\ell e_4,	&	\ad(e_9)e_4=\ell e_3,  \\
\ad(e_9)e_5=-me_6, &	\ad(e_9)e_6=me_5,	&	\ad(e_9)e_7=0, 		& \ad(e_9)e_8=0.
\end{array}
\end{equation}
Thus,  the isotropy action of $\fr{h}$ on $\fr{m}$ yields the following subalgebra of $\fr{so}(8)=\fr{so}(\fr{m})$:
\[
\chi_{*}(\fr{h})=\ad(\fr{h})|_{\fr{m}} =\left\{\left( \:
\begin{array}{*{8}{c}}
 \cline{1-2}
  \multicolumn{1}{|c}{0}  &   \multicolumn{1}{c|}{kx}             &                                      &                                  &                    &  & &  \\
  \multicolumn{1}{|c}{-kx}         &    \multicolumn{1}{c|}{0}    &                                      &                                   &                      &  & &    \\
    \cline{1-2}
  \cline{3-4}
                                &                                     &  \multicolumn{1}{|c}{0}        &   \multicolumn{1}{c|}{\ell x}            &                     &   & &   \\
                                &   				    &     \multicolumn{1}{|c}{-\ell x}              & \multicolumn{1}{c|}{0 }  &                      &  & &       \\
                                \cline{3-4}
                                \cline{5-6}
                                 &   				    & 		   			     &                                   & \multicolumn{1}{|c}{0 } &  \multicolumn{1}{c|}{mx}  & & \\
                                & 		  		    & 					    &					&\multicolumn{1}{|c}{ -mx} &  \multicolumn{1}{c|}{0} & &  \\
                                                                \cline{5-6}
                                                                \cline{7-8}
                               &   				    & 		   			     &                                   &       & &     \multicolumn{1}{|c}{0 } &  \multicolumn{1}{c|}{0 }   \\
                                & 		  		    & 					    &					&       &   &  \multicolumn{1}{|c}{0} &  \multicolumn{1}{c|}{0}   \\
                                \cline{7-8}
\end{array}\right) : x\in\bb{R}\right\}.
\]
By using the  basis of $\fr{spin}(7)\subset\fr{so}(8)$ given  in terms of the skew-symmetric matrices $E_{ij}$  (cf.~e.g.~\cite{BK}),  we see that
a Cartan subalgebra of $\fr{spin}(7)$ which occurs as the lift of a Cartan subalgebra of $\fr{so}(7)$ has the following expression:
\[
\mathsf{t}^{3}=\left\{\left( \:
\begin{array}{*{8}{c}}
 \cline{1-2}
  \multicolumn{1}{|c}{0}  &   \multicolumn{1}{c|}{x}             &                                      &                                  &                    &  & &  \\
  \multicolumn{1}{|c}{-x}         &    \multicolumn{1}{c|}{0}    &                                      &                                   &                      &  & &    \\
    \cline{1-2}
  \cline{3-4}
                                &                                     &  \multicolumn{1}{|c}{0}        &   \multicolumn{1}{c|}{ y }            &                     &   & &   \\
                                &   				    &     \multicolumn{1}{|c}{ -y}              & \multicolumn{1}{c|}{0 }  &                      &  & &       \\
                                \cline{3-4}
                                \cline{5-6}
                                 &   				    & 		   			     &                                   & \multicolumn{1}{|c}{0 } &  \multicolumn{1}{c|}{ z}  & & \\
                                & 		  		    & 					    &					&\multicolumn{1}{|c}{ -z} &  \multicolumn{1}{c|}{0} & &  \\
                                                                \cline{5-6}
                                                                \cline{7-8}
                               &   				    & 		   			     &                                   &       & &     \multicolumn{1}{|c}{0} &  \multicolumn{1}{c|}{-(x-y-z) }   \\
                                & 		  		    & 					    &					&       &   &  \multicolumn{1}{|c}{x-y-z } &  \multicolumn{1}{c|}{0}   \\
                                \cline{7-8}
\end{array}\right) : x, y, z\in\bb{R}\right\}\subset\fr{spin}(7).
\]
By comparing $\chi_{*}(\fr{h})$ with $\mathsf{t}^{3}$, we see that  $\chi_{*}(\fr{h})$ is contained in $\mathsf{t}^{3}$ if and only if  $k-\ell-m=0$.
Consequently, $C_{k, \ell, m}$ admits invariant $\Spin(7)$-structures if and only if  $k=\ell+m$.
\end{proof}

Explicit examples of invariant $\Spin(7)$-structures on $C_{\ell+m, \ell, m}$ will be given in Section \ref{CklmExamples}.
In the next two examples, we describe some special spaces belonging to the family $C_{k, \ell, m}$.

\begin{example}\label{su2su2s2}
For $k=\ell=1$ and $m=0$, the space $C_{1,1,0}$ coincides with the direct product $(\SO(4)/\SO(2))\times\SU(2)=\bb{V}_{4, 2}\times\Ss^{3}$, where we recall that $\bb{V}_{4, 2}\cong\Ss^{3}\times\Ss^{2}$ (cf.~\cite{Nik2}).
Here, the 1-dimensional Lie subalgebra $\fr{u}(1)_{1, 1}\subset\fr{t}^{2}$, where $\fr{t}^{2}$ is a maximal torus of $\SO(4)$, corresponds to the diagonal embedding of $\fr{u}(1)$ in $\fr{so}(4)$,
and the existence of an invariant $\Spin(7)$-structure  follows from the inclusions $\fr{u}(1)_{1, 1}\subset\fr{t}^{2}\subset\mathsf{t}^{3}$.
\end{example}

\begin{example}\label{C100}
Consider the direct product of the group manifold $\SU(2)\times\SU(2)=\Spin(4)$ with the homogeneous space $\Ss^2=\SU(2)/\U(1)$,
\[
M=G/H=G'/H'\times G''/H''=\displaystyle\frac{\SU(2)\times\SU(2)}{\{e\}}\times\frac{\SU(2)}{\U(1)}\,.
\]
The stability algebra  is given by $\fr{h}\coloneqq\{0\}\oplus\fr{u}(1)\cong\fr{u}(1)$, and the corresponding isotropy action is effective.
In particular, $\chi_{*}(\fr{u}(1))$ acts trivially on the tangent space of $G'/H'=\SU(2)\times\SU(2)$, while $\U(1)$ sits diagonally inside $\SU(2)$
and induces an irreducible representation when restricted to $T_{eH''}G''/H''$.
Thus, this manifold belongs to the family $C_{k, \ell, m}$ for $k=1$ and $\ell=m=0$, and we have the obvious diffeomorphisms
\[
 \frac{\SU(2)\times\SU(2)}{\{e\}}\times\frac{\SU(2)}{\U(1)}=C_{1, 0, 0}\cong C_{0, 1, 0}\cong C_{0, 1, 0}\,.
\]
The reductive decomposition of $\fr{g}=\fr{su}(2)\oplus\fr{su}(2)\oplus\fr{su}(2)$ is
\[
\fr{g} \cong \fr{so}(4)\oplus\fr{su}(2)=\fr{u}(1)\oplus\fr{m}, \quad \fr{m}=\fr{n}\oplus V^2=\fr{n}_1\oplus\fr{n}_2\oplus \U,
\]
where $\fr{so}(4)\cong\fr{su}(2)\oplus\fr{su}(2)=\fr{n}_1\oplus\fr{n}_2=\fr{n}\cong T_{e}\Spin(4)$, and $\U\coloneqq V^{2}=[\bb{C}]_{\bb{R}}$ denotes the realification of the standard representation of $\U(1)$ on $\bb{C}$.
Since the triple $(1,0,0)$ does not satisfy the condition $k=\ell+m$, the space $M=G'/H'\times G''/H''$ cannot admit any invariant $\Spin(7)$-structure by Proposition \ref{mi1}.
\end{example}

\medskip
\subsection{The Calabi-Eckmann manifold  ${(\SU(3)/\SU(2))\times\SU(2)}$}\label{ss5ss3}
This space is the direct product of the homogeneous spaces $\SU(3)/\SU(2)$ and $\SU(2)$.
The former is the canonical presentation of the unit 5-sphere $\Ss^5\subset\bb{C}^3$ acted on transitively and almost effectively by the Lie group $\SU(3)$ with stability group at $(1, 0, 0)\in\bb{C}^{3}$ given by
\[
\left\{
\begin{pmatrix}
1 & 0 \\
0 &  A
 \end{pmatrix}\in \SU(3) : A\in\SU(2)  \right\}\cong \SU(2).
\]
Now, the isotropy group $\chi(\SU(2))$ lies inside $\SU(3)$. 
Recall that the latter is  the  stability  subgroup of a unit  vector $v \in \mathbb{R}^7$  in  the $\G_2$-module $\mathbb{R}^7$    and  $\G_2$  is  the  stability subgroup   
of  a unit vector  $w \in \mathbb{R}^8$  of  the  tautological $\Spin(7)$-module $\mathbb{R}^8$ (see e.g.~\cite{Br1,Rei}). In particular, $\G_2$ is a maximal subgroup of $\Spin(7)$. 
Consequently, the homogeneous space $M=(\SU(3)/\SU(2))\times\SU(2)$ admits invariant $\Spin(7)$-structures.
Explicit examples will be discussed in Section \ref{CEck}.

\begin{remark}\label{CEck}
The space $(\SU(3)/\SU(2))\times\SU(2)$ is a {\it Calabi-Eckmann manifold},  i.e.,
a complex homogeneous non-K\"ahler manifold diffeomorphic to the product of two odd-dimensional spheres of dimension greater than two.
In particular, it is a torus bundle over $\bb{CP}^{2}\times\bb{CP}^1$,
\[
{\rm T}^{2}\cong\frac{\U(2)\times\U(1)}{\SU(2)}\longrightarrow  \frac{\SU(3)\times\SU(2)}{\SU(2)\times\{e\}}\longrightarrow \frac{\SU(3)\times\SU(2)}{\U(2)\times\U(1)}\,,
\]
and, consequently, a C-space with $H^{2}(M,\bb{Z})=0$ (see \cite{Klaus}).
\end{remark}

\medskip
\subsection{The space  $(\SU(2) \times \SU(2) \times \SU(2))/\Delta(\SU(2)))\times(\SU(2)/\U(1))$}\label{Space4}
For the sake of convenience, from now on we let
\[
L^6 \coloneqq \displaystyle\frac{\SU(2) \times \SU(2) \times \SU(2)}{\Delta(\SU(2))},  \quad   X^6 \coloneqq \displaystyle\frac{\SO(3)\times\SO(3)\times\SO(3)}{\Delta\SO(3)}.
\]
Both manifolds $L^6$ and $X^6$ are {\it Ledger-Obata spaces}, i.e., of the form $(K \times K \times K)/\Delta K$,  with $K$ a compact simple Lie group and $\Delta K=\{(k, k, k) : k\in K\}$.
Moreover, there is a natural isomorphism between the compact homogeneous space $(K\times K\times K)/\Delta K$ and the compact semisimple Lie group  $K\times K$.
Consequently, the corresponding 8-manifolds are diffeomorphic to  $C_{1, 0, 0}=\Spin(4)\times\Ss^{2}$.
Notice however that  $L^6 \times(\SU(2)/\U(1))$ does not belong to the family $C_{k, \ell, m}$.

Since $\SU(2)/\bb{Z}_2 \cong \SO(3)$, the effective coset  $X^6 \times (\SO(3)/\SO(2))$ is covered by the almost effective simply connected coset $L^6 \times(\SU(2)/\U(1))$.
Thus, to conclude the proof of Theorem \ref{MainTheoremB}, it is sufficient to show that  the space $X^6 \times(\SO(3)/\SO(2))$ does not admit any  invariant $\Spin(7)$-structure.
To this aim, we will first describe the isotropy representation, and then the space of invariant forms.

Let $\tilde{\fr{k}}  =\fr{k}\oplus\fr{k}\oplus\fr{k}$ and $\Delta\fr{k}=\{(X, X, X) : X\in\fr{k} \}$ be the Lie algebras of $K \times K \times K$ and $\Delta K$, respectively.
A natural  choice of an $\Ad(\Delta K)$-invariant complement of $
\Delta\fr{k}$ in $\tilde{\fr{k}}$ is  given for instance  by (see e.g. \cite{Nikol})
\[
\fr{n}=\left\{ \left(a_1X, a_2X, a_3X\right) \in \tilde{\fr{k}}~:~X\in\fr{k}, \ a_i\in\bb{R}, \  \sum_{i=1}^{3}a_i=0\right\},
\]
and then $\Delta\fr{k}=\left\{(a X, a X, a X)  \in \tilde{\fr {k}}~:~X\in\fr{k}, \  a \in \bb{R} \right\}$.
In our case, $\fr{k} = \fr{so}(3)\cong\fr{su}(2)$.

Consider on $\tilde{\frak {k}}$  the  bi-invariant metric  $\langle A,  B \rangle=-(1/2){\rm tr}(AB)$, which is a multiple of the corresponding Killing form.
The Lie algebra  $\fr{so}(3)$  can be identified with the span of $\{E_{12}, E_{13}, E_{23}\}$, and
the matrices
\[
h_1 \coloneqq E_{12}+E_{45}+E_{78}, \quad h_{2} \coloneqq E_{13}+E_{46}+E_{79}, \quad h_3 \coloneqq E_{14}+E_{56}+E_{89}
\]
generate $\Delta\fr{so}(3)\cong\fr{so}(3)$.
By using the Gram-Schmidt process, we see that an orthogonal splitting of $\fr{n}$ with respect to  $\langle \cdot , \cdot \rangle$  is  given by
\[
\fr{n}=\fr{n}_1\oplus\fr{n}_2= \Big\{a_1(X, 0, -X)~:~X\in\fr{k},~a_1\in\bb{R} \Big\}\oplus \left \{ a_2\left(-\frac{1}{2}X, X, -\frac{1}{2}X\right)~:~X\in\fr{k},~a_2\in\bb{R}\right \},
\]
where both $\fr{n}_1$ and $\fr{n}_2$ are  irreducible.
Therefore
\[
\renewcommand\arraystretch{1.4}
\begin{array}{lll}
e_{1}=\frac{1}{\sqrt{2}}(E_{12}-E_{78}), & e_2=\frac{1}{\sqrt{2}}(E_{13}-E_{79}), &   e_3=\frac{1}{\sqrt{2}}(E_{23}-E_{89}) \\
e_{4}=-\frac{\sqrt{6}}{6}(E_{12}-2E_{45}-E_{78}), & e_{5}=-\frac{\sqrt{6}}{6}(E_{13}-2E_{46}-E_{79}), & e_{6}=-\frac{\sqrt{6}}{6}(E_{23}-2E_{56}-E_{89})\\
\end{array}
\]
form a  $\langle \cdot , \cdot\rangle$-orthonormal basis of $\fr{n}$  such that $\fr{n}_1=\Span_{\bb{R}}\{e_1, e_2, e_3\}$ and $\fr{n}_2=\Span_{\bb{R}}\{e_4, e_5, e_6\}$, respectively.
Let us denote by $\Ee$ the standard representation of $\fr{so}(3)$. Then,  one  may  identify $\fr{n}_1=\E$ and $\fr{n}_2=\E'$,  where   $\Ee'$  denotes another  copy of $\Ee$.

Set now  $\fr{g} \coloneqq \fr{k}\oplus\fr{k}\oplus\fr{k}\oplus\fr{k}=3\fr{k}\oplus\fr{k}=4\fr{so}(3)=4\fr{su}(2)$,
and consider  the Lie algebra
\[
\fr{h} \coloneqq \Delta\fr{so}(3)\oplus\fr{u}(1)\,.
\]
This is a subalgebra of $\fr{so}(8)$, since    $C_{\SO(8)}(\Delta\fr{so}(3))= \U(1)\times\U(1)$. Its defining representation is $\Ee\oplus\Ee'\oplus\U$, where $\U$ denotes the standard representation
of $\fr{u}(1)$ on $\bb{C}$.
Since $\fr{u}(1)$ sits inside the last summand of $\fr{g}$, $\fr{h}$ sits inside $\fr{g}$ and  the pair $(\fr{g}, \fr{h})$ induces the  8-dimensional effective homogeneous space
$X^6 \times\frac{\SO(3)}{\SO(2)}$.
Summing up, a reductive decomposition of $\fr{g}=4\fr{so}(3)$ is
\[
\fr{g}= \fr{h}\oplus\fr{m}, \quad \fr{m} = \fr{n}\oplus\U = \fr{n}_1\oplus\fr{n}_2\oplus\U = \E\oplus\E'\oplus\U,
\]
where $\fr{n}$ coincides with the tangent space to  $X^6$ and $\U$ with the tangent space  to  $\frac{\SU(2)}{\U(1)}$ at the identity coset.
Moreover,  an orthonormal basis of $\fr{m}$ is given by
\[
\left\{e_1,\ldots,e_6, e_7 \coloneqq -E_{10, 11},~e_8\coloneqq-E_{10, 12}\right\},
\]
so that  $\U=\Span_{\bb{R}}\{e_7, e_8\}$.
The elements $\{h_1, h_2, h_3\}$ defined above, together with $h_4 \coloneqq E_{11, 12}$ generate the stability algebra $\fr{h}$, which acts on $\fr{m}$ via the isotropy representation as follows
\begin{equation}\label{isoObata}
\begin{array}{lll}
\chi_{*}(h_1)e_1=0, &   \chi_{*}(h_2)e_1=-e_3, & \chi_{*}(h_3)e_1=e_2,		\\
\chi_{*}(h_1)e_2=e_3, & \chi_{*}(h_2)e_2=0, & \chi_{*}(h_3)e_2=-e_{1},		\\
\chi_{*}(h_1)e_3=-e_2, & \chi_{*}(h_2)e_3=e_1, & \chi_{*}(h_3)e_3=0,		\\
\chi_{*}(h_1)e_4=0, & \chi_{*}(h_2)e_4=-e_6,  & \chi_{*}(h_3)e_4=e_5,		\\
\chi_{*}(h_1)e_5=e_6, & \chi_{*}(h_2)e_5=0, & \chi_{*}(h_3)e_5=-e_4,		\\
\chi_{*}(h_1)e_6=-e_5, & \chi_{*}(h_2)e_6=e_4, & \chi_{*}(h_3)e_6=0,		\\
\end{array}
\end{equation}
and
\begin{equation}\label{isoS2}
\chi_{*}(h_4)e_7=-e_8,\quad	\chi_{*}(h_4)e_8=e_7.
\end{equation}

We can now determine the invariant forms.
\begin{lemma}  \label{lemmaspace3invforms}
Let $X^6\times (\SO(3)/\SO(2)) = G/H$. Then, the following hold:
\begin{enumerate}[1)]
\item  the space of $G$-invariant 1-forms is trivial;
\item the space of  $G$-invariant  2-forms is 2-dimensional and it is generated by $\omega_{1} \coloneqq e^{14}+e^{25}+e^{36}$  and $\omega_2 \coloneqq e^{78}$;
\item the space of  $G$-invariant 3-forms is 2-dimensional and it is generated by $\{e^{123}, e^{456}\}$;
\item  the space of  $G$-invariant 4-forms 2-dimensional and it is generated by $\left\{\omega_1\wedge\omega_1, \omega_1\wedge\omega_2\right\}$.
\end{enumerate}
\end{lemma}
\begin{proof}
The assertion for invariant 1-forms immediately follows from \eqref{isoObata} and \eqref{isoS2}.
For invariant 2-forms on $X^6$, we have the equivariant decomposition
\[
(\Lambda^{2}\fr{n}^*)^{H}=(\Lambda^{2}\fr{n}_1^*)^{H}\oplus (\fr{n}_1^*\wedge\fr{n}_2^*)^{H}\oplus(\Lambda^{2}\fr{n}_2^*)^{H}\,.
\]
Here, the first and the third module vanish, while from (\ref{isoObata}) we see that the second module is generated by $\omega_1$.
For the invariant 3-forms on $X^6$, we see that only the modules $(\Lambda^{3}\fr{n}_1^*)^{H}$ and $(\Lambda^{3}\fr{n}_2^*)^{H}$ are non-trivial.
In particular, the 9-dimensional spaces $\Lambda^2\fr{n}_1^*\wedge\fr{n}_2^*$ and $\fr{n}_1^*\wedge\Lambda^{2}\fr{n}_2^*$ do not contain any invariant element.
Now, the claim for $G/H$ follows  from the orthogonal decompositions
\begin{eqnarray*}
(\Lambda^{3}\fr{m}^*)^H	&=&	(\Lambda^{3}\fr{n}^*)^H\oplus(\Lambda^{2}\fr{n}^*\wedge\U^{*})^H\oplus(\fr{n}^*\wedge\Lambda^{2}\U^*)^{H}, \\
(\Lambda^{4}\fr{m}^*)^H	&=&	(\Lambda^{4}\fr{n}^*)^H\oplus (\Lambda^{3}\fr{n}^*\wedge\U^*)^H\oplus (\Lambda^{2}\fr{n}^*\wedge\Lambda^2\U^*)^{H}.
 \end{eqnarray*}
\end{proof}

\begin{prop}\label{ana}
The space $X^6 \times(\SO(3)/\SO(2))$ cannot admit any invariant $\Spin(7)$-structure.
\end{prop}
\begin{proof}
We show that there are no invariant admissible 4-forms on this space.
Let us consider the generic $\mathrm{Ad}(H)$-invariant 4-form on $\fr{m}$
\[
\Phi \coloneqq a_1(e^{1245}+e^{1346}+e^{2356})+a_2(e^{1478}+e^{2578}+e^{3678}),
\]
where $a_1, a_2\in\bb{R}$.
By   \cite[Thm.~4.3.3]{Kar}, if $\Phi$ defines a $\Spin(7)$-structure inducing the orientation $e^{12345678}$, then the norm of a vector $ u=\sum_{i=1}^{8} u^{k}e_{k} \in \fr{m}$ with $u^{1}\neq 0$ is proportional to the
determinant of the $7\times7$ matrix $\left( a_{ij}\right)$ with the following entries
\[
a_{ij} \coloneqq (e_{i}\lrcorner u\lrcorner\Phi)\wedge (e_{j}\lrcorner u\lrcorner\Phi)\wedge(u\lrcorner\Phi)(e_2, \ldots, e_8), \quad 2\leq i \leq j \leq8.
\]
Considering $u=e_1$, an easy computation shows that all $a_{ij}$ vanish. Thus, $\Phi$ cannot define an invariant $\Spin(7)$-structure.
\end{proof}


\section{Explicit examples  of  invariant $\Spin(7)$-structures}\label{genericexa}

\subsection{The infinite family $C_{k, \ell, m}$}\label{CklmExamples}
By Proposition \ref{mi1}, the space  $C_{k, \ell, m}$ admits invariant $\Spin(7)$-structures if and only if  $k=\ell+m$.
In particular, since $\ell\geq m\geq 0$ with $\ell>0$, two different cases arise, namely
\[
C_{\ell+m, \ell, m}, \mbox{ with } \ell>m>0, \qquad C_{\ell, \ell, 0}, \mbox{ with } \ell>0.
\]
For these two classes of homogeneous manifolds, the invariant objects (metrics, forms) are different.

We begin examining the family $C_{k,\ell,m} =G/H$ for  $k>\ell>m>0$, where $G = \SU(2)\times\SU(2)\times\SU(2)$ and $H = \U(1)_{k, \ell, m}$.
In our computations, we shall keep on using the notation  introduced in Section \ref{cklm}.

Let $\fr{m}_{0} \coloneqq \fr{h}^{\perp}$ denote the orthogonal complement of $\fr{h}$ inside the maximal torus $\fr{t}$ with respect to $\langle \cdot , \cdot \rangle$.
The space $\fr{m}_{0}$ is spanned by $\{e_7, e_8\}$ and it is an abelian Lie algebra,  i.e., $[e_7, e_8]=0$.
Thus, whenever $k, \ell, m\in\bb{Q}$, it generates a closed connected  2-dimensional abelian subgroup of ${\rm T}^{3}$, i.e., a 2-torus, which we denote by ${\rm T}^{2}_{k, \ell, m}\subset{\rm T}^{3}$.
This induces the 7-dimensional homogeneous space (cf. \cite{Nik2, Rei})
\[
Q_{k, \ell, m}=(\SU(2)\times\SU(2)\times\SU(2))/{\rm T}^{2}_{k, \ell, m}\cong \Ss^{3}\times\Ss^{2}\times\Ss^{2}\,,
\]
which is a  circle bundle $\mathsf{q}  : Q_{k, \ell, m}\to \Ss^{2}\times\Ss^{2}\times\Ss^{2}$ over $N^{6} \coloneqq \SU(2)\times\SU(2)\times\SU(2)/{\rm T}^{3}=\Ss^2\times\Ss^2\times\Ss^2$.

Moving to the family $C_{k, \ell, m}$,  the reductive decomposition  described before now reads
\begin{equation}\label{mC}
\fr{g}=\fr{u}(1)_{k, \ell, m}\oplus\fr{m}, \quad  \fr{m}\cong T_{o}C_{k, \ell, m}=\fr{m}_1\oplus\fr{m}_2\oplus\fr{m}_3\oplus\fr{m}_{0}=\fr{p}\oplus\fr{m}_{0}
\end{equation}
where $ \fr{m}_1=\Span_{\bb{R}}\{e_1, e_2\}$, $\fr{m}_2=\Span_{\bb{R}}\{e_3, e_4\},$ $\fr{m}_3=\Span_{\bb{R}}\{e_5, e_6\}$
and  $\fr{p} \coloneqq \fr{m}_1\oplus\fr{m}_2\oplus\fr{m}_3$ coincides with the 6-dimensional tangent space of $\Ss^{2}\times\Ss^{2}\times\Ss^{2}$.
In the following, every $\fr{m}_i$,  $i=0, 1, 2, 3$,  will be viewed as an $\Ad(H)$-module.

\begin{remark}\label{my11}
When $ {\rm gcd}(k, \ell, m)=1,$  the simply connected coset $C_{k, \ell, m}$ is a non-K\"ahler C-space in the sense of Wang (cf.~\cite{ACh, Pod}),
i.e., a simply connected compact homogeneous complex manifold which is not K\"ahler.
Indeed, since the integers $k, \ell, m$ are assumed to be relatively prime, the  action  of $H$ on  $G$
 \[
 (e^{i\phi}, (x, y, z))\mapsto (e^{ik\phi}x, e^{i\ell\phi}y, e^{im\phi}z),
 \]
 is free.  Consequently, the projection $\mathsf{c} :C_{k, \ell, m}\to N^{6}=\Ss^{2}\times\Ss^{2}\times\Ss^{2}$  defines a principal torus bundle,
\[
{\rm T}^{2}_{k, \ell, m}\cong{\rm T}^{3}/H  \  \longrightarrow  \ C_{k, \ell, m}=G/H \ \stackrel{\mathsf{c}}{\longrightarrow}  \  N^{6}=G/{\rm T}^{3}=\Ss^{2}\times\Ss^{2}\times\Ss^{2},
\]
and  we have  the following diagram:
 \[
\xymatrix{
\U(1)_{k, \ell,  m} \ar[d]  \ar @{^{(}->}[r]   \ar@[][dr]  & \   {\rm T}^{3}               \ar[d]                                            & {\rm T}^{2}_{k, \ell, m} \ar @{_{(}->}[l] \ar@[][ld] \ar[d] \\
Q_{k, \ell, m}\cong \Ss^{3}\times\Ss^{2}\times\Ss^{2} \ar[rd]_{\mathsf{q}}              & G=(\SU(2))^{\times 3} \ar@[][l] \ar[d]^{\pi} \ar@[][r] & C_{k, \ell, m}\cong\Ss^{3}\times\Ss^{3}\times\Ss^{2} \ar[ld]^{\mathsf{c}}  \\
                                             & N^{6}=G/{\rm T}^{3}\cong \Ss^{2}\times\Ss^{2}\times\Ss^{2} &
}
\]
Since $C_{k, \ell, m}$ is a  ${\rm T}^2_{k, \ell, m}$-bundle over the (full) flag manifold $N^{6}=\bb{CP}^{1}\times\bb{CP}^{1}\times\bb{CP}^{1}$, it must be a C-space.
Any complex structure $J_{0}$ on $\fr{m}_{0}$ is automatically $\Ad(H)$-invariant. Hence, whenever $J_{\fr{p}}$ defines an $\Ad({\rm T}^{3})$-invariant complex structure on
$N^{6}=G/{\rm T}^{3}$, the endomorphism $J_{\fr{m}}\coloneqq J_{\fr{p}}+J_{\fr{m}_{0}}$ defines an $\Ad(H)$-invariant complex structure on $C_{k, \ell, m}$.
 \end{remark}

In order to describe the set of $G$-invariant metrics on $C_{k, \ell, m}$, or equivalently, the space of $\Ad(H)$-invariant inner products on $\fr{m}$,
we first need to examine the properties of the modules $\fr{m}_i, i = 0, \ldots, 3$.
\begin{lemma}\label{men}
Assume that $k>\ell> m>0$. Then, the $\fr{h}$-modules $\fr{m}_1, \fr{m}_2$ and $\fr{m}_3$ are pairwise inequivalent and irreducible.
In contrast, $\fr{m}_{0}$ decomposes into two irreducible 1-dimensional submodules, which we denote by $\fr{m}_4\coloneqq\Span_{\bb{R}}\{e_7\}$ and $\fr{m}_5\coloneqq\Span_{\bb{R}}\{e_8\}$, respectively.
These submodules are mutually equivalent, and they are not equivalent to $\fr{m}_1, \fr{m}_2, \fr{m}_3$. Consequently, the orthogonal $\Ad(H)$-invariant decomposition
\[
\fr{m}=\fr{m}_1\oplus\fr{m}_2\oplus\fr{m}_3\oplus\fr{m}_0=\fr{m}_1\oplus\fr{m}_2\oplus\fr{m}_3\oplus\fr{m}_4\oplus\fr{m}_5
\]
is not in general unique.
\end{lemma}
\begin{proof}
By (\ref{adjon1}), we see that the weights of the adjoint action of $H$ on $\fr{m}\cong\bb{R}^{8}$ are
\[
\left(\exp \left(\pm2k \sqrt{-1}\varphi\right),~\exp\left(\pm2\ell \sqrt{-1}\varphi\right),~\exp\left(\pm 2m\sqrt{-1}\varphi\right),~1,~1\right), \quad \varphi\in[0, 2\pi].
\]
Since the stability group $H$ acts with different weights on  $\fr{m}_1, \fr{m}_2$, $\fr{m}_3,$ whenever $k\neq\ell\neq m\neq0$,
it follows that $\fr{m}_1$, $\fr{m}_2$, and $\fr{m}_3$ are mutually inequivalent.
Now, although $\fr{m}_{0}$ is irreducible under the adjoint action of ${\rm T}^{2}_{k, \ell, m}$, it decomposes into two equivalent $\Ad(H)$-invariant submodules, which are generated by $e_7$ and $e_8$, respectively.
Moreover, since $[\fr{h}, \fr{m}_{0}]=0$, we see that $\Ad(H)|_{\fr{m}_{0}}=\Id$. Hence, $\fr{m}_4$ and $\fr{m}_5$ can be replaced by any pair of orthogonal 1-dimensional submodules in $\fr{m}_{0}$.
\end{proof}

Fix some angle $\theta\in [0, 2\pi]$, consider the vectors
\[
e^{\theta}_7\coloneqq\cos(\theta) \, e_7+\sin(\theta) \,  e_8, \quad e^{\theta}_8\coloneqq-\sin(\theta) \,  e_7+\cos(\theta) \,  e_8,
\]
and let $\fr{m}_4^{\theta}\coloneqq\mathrm{span}_{\bb{R}}\{e^\theta_7\}$, $\fr{m}_5^{\theta}\coloneqq\mathrm{span}_{\bb{R}}\{e^\theta_8\}$.
Then, the $H$-module $\fr{m}_{0}^{\theta}\coloneqq \fr{m}_4^{\theta}\oplus\fr{m}_5^{\theta}$ is equivalent to $\fr{m}_{0}=\fr{m}_4\oplus\fr{m}_5$, and
the pair $\left\{e_7^\theta,e_8^\theta\right\}$ is an orthonormal basis of $(\fr{m}^{\theta}_0, \langle \cdot , \cdot \rangle)$.
Now, as a consequence of  Lemma \ref{men}, we obtain the following.
\begin{prop}\label{nik}  Assume that $k>\ell> m>0$.  Let  $Q$ be the normal metric on $C_{k, \ell, m}$ induced by the bi-invariant metric $\langle A, B \rangle=-2\tr(AB)$ on $\fr{g}$.
Then, up to scaling, any $G$-invariant metric on $C_{k, \ell, m}$ is given by
\[
 ( \ , \ )_{y_1,y_2,y_3,y_4, y_5, \theta} \coloneqq y_{1}^{2}\, Q|_{\fr{m}_1}+y_{2}^{2}\, Q|_{\fr{m}_2}+y_{3}^{2}\, Q|_{\fr{m}_3}+y_{4}^{2}\, Q|_{\fr{m}_4^{\theta}}+y_{5}^{2}\, Q|_{\fr{m}_5^{\theta}},
\]
for some positive real numbers $y_1, \ldots, y_5$. Thus, the space of $G$-invariant metrics on $C_{k, \ell, m}$, with $k>\ell>m>0$, is 6-dimensional.
\end{prop}

Let us now examine the spaces of invariant differential forms on $C_{k,\ell,m}$.
By the general theory of C-spaces described in \cite{ACh}, we can construct  an equivariant isomorphism
\[
\phi :\fr{u}(1)_{k, \ell, m}=Z_{\fr{g}}(\fr{u}(1)_{k, \ell, m})\to H^{2}(M, \bb{R}),
\]
from which we deduce that $H^{2}(C_{k, \ell, m}, \bb{Z})\cong\bb{Z}$. This also follows from the exact sequence  (see \cite{ACh, Klaus})
\[
H^{1}({\rm T}^{2}, \bb{Z})\cong\bb{Z}^{2}\longrightarrow H^{2}(N^{6}, \bb{	Z})\cong\bb{Z}^{3}\longrightarrow H^{2}(M, \bb{Z}).
\]
In more geometric terms,  any element $\xi\in \fr{m}_{0}+C_{\fr{p}}(\fr{h})$, where $C_{\fr{p}}(\fr{h})$ denotes the centralizer of  the stability subalgebra $\fr{h}$
in $\fr{p}=T_{\mathsf{c}(o)}N^{6}=T_{\mathsf{c}(o)}G/{\rm T}^{3}$, induces  an invariant 1-form form $\tilde{\xi}$ and moreover an exact invariant 2-form $\dd\tilde\xi$ on $C_{k, \ell, m}$.
Since $C_{\fr{p}}(\fr{h})$ is trivial,  there is  a bijection
\[
\xi\in \fr{m}_{0} \ \longmapsto \ \tilde{\xi}\in(\Lambda^{1}\fr{m}_{0}^{*})^{H}\cong (\fr{m}_{0}^{*})^{H}\cong\fr{m}_{0}^{H}.
\]
 This means that $ \chi_{*}(e_9)(e^7)=0,$  $\chi_{*}(e_9)(e^{8})=0$, i.e., the  dual 1-forms $e^{7}, e^{8}$ of  $e_7, e_8$  induce $G$-invariant 1-forms on $C_{k, \ell, m}$.  The same holds true for
the  dual 1-forms
\[
e_{\theta}^{7}=\cos(\theta) \,  e^{7}+\sin(\theta)\, e^{8}, \quad e_{\theta}^{8}=-\sin(\theta) \, e^{7}+\cos(\theta)\, e^{8},
\]
and their wedge product satisfies the relation  $e_{\theta}^{7}\wedge e_{\theta}^{8}=e^{7}\wedge e^{8}$. On the other hand,  (\ref{adjon1}) gives
\begin{eqnarray*}
\chi_{*}(e_9)e^{1}&=&\sum_{j=1}^{8}\big(\chi_{*}(e_9)e^{1}\big)(e_j)e^{j}=-\sum_{j}e^{1}\big(\chi_{*}(e_9)e_{j}\big)e^{j}=-e^{1}\big(\chi_{*}(e_9)e_2\big)e^{2}=-ke^{2},
\end{eqnarray*}
and, similarly,
\begin{equation}
\chi_{*}(e_9)e^{2}=ke^{1}, \ \chi_{*}(e_9)e^{3}=-\ell e^{4}, \  \chi_{*}(e_9)e^{4}=\ell e^{3}, \  \chi_{*}(e_9)e^{5}=-m e^{6}, \  \chi_{*}(e_9)e^{6}=me^{5}.
\end{equation}
Thus, a basis for the space of invariant 1-forms on $C_{k, \ell, m}$, when $k>\ell>m>0$, is given by $\{e^7,e^8\}.$

For the  invariant 2-forms on $C_{k, \ell, m}$, we obtain the following.
\begin{lemma}\label{2forms}
If the integers $k, \ell, m$ satisfy $k>\ell> m>0$, then the space of $G$-invariant 2-forms on $C_{k, \ell, m}$  is 4-dimensional and it is generated by the 2-forms $\{e^{12}, e^{34}, e^{56}, e^{78}\}$.
\end{lemma}
\begin{proof}
By the orthogonal decomposition $  \fr{m}=\fr{p}\oplus\fr{m}_{0}=\fr{m}_1\oplus\fr{m}_2\oplus\fr{m}_3\oplus\fr{m}_0$, one  obtains the orthogonal decomposition
\[
(\Lambda^{2}\fr{m}^{*})^{H}=\big(\Lambda^{2}(\fr{p}\oplus\fr{m}_{0})^{*}\big)^{H}=(\Lambda^{2}\fr{p}^{*})^{H}\oplus(\fr{p}^{*}\wedge\fr{m}_{0}^{*})^{H}\oplus(\Lambda^{2}\fr{m}_{0}^{*})^{H}\,,
\]
where
$(\Lambda^{2}\fr{p}^{*})^{H}=(\Lambda^{2}\fr{m}_{1}^{*})^{H}\oplus(\Lambda^{2}\fr{m}_{2}^{*})^{H}\oplus(\Lambda^{2}\fr{m}_{3}^{*})^{H}\oplus(\fr{m}_1^{*}\wedge\fr{m}_2^{*})^{H}\oplus(\fr{m}_1^{*}\wedge\fr{m}_{3}^{*})^{H}
\oplus(\fr{m}^{*}_2\wedge\fr{m}^{*}_3)^{H}\,.$
It is easy to see that the 2-forms  $e^{12}$, $e^{34}$, $e^{56}$, and  $e^{78}$ are $\Ad(H)$-invariant. In particular
\[
(\Lambda^{2}\fr{m}_{1}^{*})^{H}=\Span_{\bb{R}}\{e^{12}\},  \, (\Lambda^{2}\fr{m}_{2}^{*})^{H}= \Span_{\bb{R}}\{e^{34}\},  \, (\Lambda^{2}\fr{m}_{3}^{*})^{H}=\Span_{\bb{R}}\{e^{56}\}, \,
(\Lambda^{2}\fr{m}_{0}^{*})^{H}=\Span_{\bb{R}}\{e^{78}\}.
\]
\noindent  Moreover, the $H$-module  $(\fr{p}^{*}\wedge\fr{m}_{0}^{*})^{H}\cong (\fr{p}\wedge\fr{m}_{0})^{H}$ vanishes if and only if the triple $(k, \ell, m)$ is non-zero. 
This follows  from the relations
\[
\begin{array}{llll}
\chi_{*}(e_9)e^{17}=-k e^{27}, 	& \chi_{*}(e_9)e^{47}=\ell e^{37},	& \chi_{*}(e_9)e^{18}=-ke^{28},		& \chi_{*}(e_9)e^{48}=\ell e^{38}, \\
\chi_{*}(e_9)e^{27}=k e^{17},  	& \chi_{*}(e_9)e^{57}=-m e^{67},	& \chi_{*}(e_9)e^{28}=ke^{18},		& \chi_{*}(e_9)e^{58}=-m e^{68}, \\
\chi_{*}(e_9)e^{37}=-\ell e^{47},	& \chi_{*}(e_9)e^{67}=m e^{57},	& \chi_{*}(e_9)e^{38}=-\ell e^{48},	& \chi_{*}(e_9)e^{68}=m e^{58}.
\end{array}
\]
For the remaining mixed terms corresponding to the modules $(\fr{m}_1^{*}\wedge\fr{m}_2^{*})$, $(\fr{m}_1^{*}\wedge\fr{m}_{3}^{*})$ and $(\fr{m}^{*}_2\wedge\fr{m}^{*}_3)$, one similarly computes
\[
\begin{array}{lll}
\chi_{*}(e_9)e^{13}=-ke^{23}-\ell e^{14},	& \chi_{*}(e_9)e^{15}=-ke^{25}-m e^{16},	& \chi_{*}(e_9)e^{35}=-\ell e^{45}-m e^{36}, \\
\chi_{*}(e_9)e^{24}=ke^{14}+\ell e^{23},	& \chi_{*}(e_9)e^{26}=ke^{16}+m e^{25},	& \chi_{*}(e_9)e^{46}=\ell e^{36}+m e^{45}, \\
\chi_{*}(e_9)e^{14}=-ke^{24}+\ell e^{13},	& \chi_{*}(e_9)e^{16}=-ke^{26}+m e^{15},	& \chi_{*}(e_9)e^{36}=-\ell e^{46}+m e^{35},  \\
\chi_{*}(e_9)e^{23}=ke^{13}-\ell e^{24},	& \chi_{*}(e_9)e^{25}=ke^{15}-m e^{26},	& \chi_{*}(e_9)e^{45}=\ell e^{35}-m e^{46}.
\end{array}
\]
Using these results, we see that as long as $k\neq\ell\neq m$  and   $k, \ell, m$ are non-zero, there are no further  $\Ad(H)$-invariant 2-forms.
\end{proof}

\begin{remark}
Up to scaling, the 2-forms $e^{12}, e^{34}, e^{56}$ are the unique $\U(1)$-invariant K\"ahler forms on the corresponding  factors $\SU(2)/\U(1)$ of the product $\Ss^{2}\times\Ss^{2}\times\Ss^{2}$.
By Proposition \ref{2forms}, it follows that the 2-form $\omega_{\fr{m}}\coloneqq\omega_\fr{p}+\omega_{\fr{m}_{0}}=e^{12}+e^{34}+e^{56}+e^{78}$
is the fundamental 2-form  associated to the invariant complex structure $J_{\fr{m}}=J_{\fr{p}}+J_{\fr{m}_{0}}$ discussed in Remark \ref{my11}, with $\omega_{\fr{m}_{0}}=e^{78} $.
Of course, $\omega_{\fr{m}}$ is not  K\"ahler, since ${\rm d} e^{78}\neq 0$ (cf.~Appendix \ref{app-detailsCklm}).
\end{remark}

We can now discuss the invariant 3-forms on $C_{k,\ell, m}$.  As we will see below,  the condition $k=\ell+m$ appears in a natural way.
\begin{lemma}\label{3forms}
Assume that the integers $k, \ell, m$, with $k>\ell>m>0$, satisfy $k-\ell-m=0$.
Then,  the space of $G$-invariant 3-forms on  $C_{\ell+m, \ell, m}$  is 8-dimensional and it  is generated by the 3-forms
\begin{equation}\label{3forms1}
e^{127}, ~e^{128},~e^{347}, ~e^{348},~e^{567},~e^{568},
\end{equation}
and
\begin{equation}\label{3forms2}
\upalpha_1\coloneqq e^{135}-e^{146}+e^{245}+e^{236}, \quad \upalpha_2\coloneqq e^{136}+e^{145}+e^{246}-e^{235}\,.
\end{equation}
\end{lemma}

\begin{proof}
Since $\fr{m}=\fr{p}\oplus\fr{m}_{0}$, we obtain
\[
(\Lambda^{3}\fr{m}^{*})^{H}=\Big(\Lambda^{3}(\fr{p}\oplus\fr{m}_0)^{*}\Big)^{H}=(\Lambda^{3}\fr{p}^{*})^{H}\oplus(\Lambda^{2}\fr{p}^{*}\wedge\fr{m}^{*}_{0})^{H}\oplus(\fr{p}^{*}\wedge\Lambda^{2}\fr{m}_{0}^{*})^{H}\,.
\]
The  3-forms  given  by (\ref{3forms1}) are invariant,   independently of   the condition $k-\ell-m=0$. They occur by considering the wedge product of the invariant 2-forms
$e^{12}, e^{34}, e^{56}$  with the elements $e^7$ and  $e^8$, which span $(\fr{m}_{0}^{*})^{H}\cong(\fr{m}_{0})^{H}=\fr{m}_{0}$.
Thus, they span the 6-dimensional  factor $(\Lambda^{2}\fr{p}^{*}\wedge\fr{m}_{0}^{*})^{H}$.
We mention that the  invariant 3-forms $e^{ij7}_{\theta}=e^{ij}\wedge e^{7}_{\theta}$ and $e^{ij8}_{\theta}=e^{ij}\wedge e^{8}_{\theta}$, for $ij\in\{12, 34, 56\},$ do not induce new forms.
Now, for generic non-zero $k, \ell, m,$ the third module $(\fr{p}^{*}\wedge\Lambda^{2}\fr{m}_{0}^{*})^{H}$  vanishes, since
\[
\begin{array}{lll}
\chi_{*}(e_9)e^{178}=-ke^{278}, 	&  \chi_{*}(e_9)e^{278}=ke^{178}, 	&  \chi_{*}(e_9)e^{378}=-\ell e^{478}, \\
\chi_{*}(e_9)e^{478}=\ell e^{378}, 	&  \chi_{*}(e_9)e^{578}=-m e^{478}, 	&  \chi_{*}(e_9)e^{678}=m e^{578}.
\end{array}
\]
Let us now consider the first factor.  Due to the orthogonal decomposition $\fr{p}=\fr{m}_1\oplus\fr{m}_2\oplus\fr{m}_3$, one has
\begin{eqnarray*}\label{DEC3P}
(\Lambda^{3}\fr{p}^{*})^{H}&=&(\Lambda^{2}\fr{m}_{1}^{*}\wedge\fr{m}_2^{*})^{H}\oplus(\Lambda^{2}\fr{m}_{1}^{*}\wedge\fr{m}_3^{*})^{H}\oplus(\Lambda^{2}\fr{m}_{2}^{*}\wedge\fr{m}_3^{*})^{H}\nonumber\\
&&\oplus(\fr{m}_1^{*}\wedge\Lambda^{2}\fr{m}_{2}^{*})^{H}\oplus(\fr{m}_1^{*}\wedge\Lambda^{2}\fr{m}_{3}^{*})^{H}\oplus(\fr{m}_2^{*}\wedge\Lambda^{2}\fr{m}_{3}^{*})^{H}
\oplus(\fr{m}_1^{*}\wedge\fr{m}_2^{*}\wedge\fr{m}_3^{*})^{H}\,.
\end{eqnarray*}
First, we show that for non-zero $k, \ell, m,$ the first 6 modules vanish. Indeed, note that
\[
\begin{array}{ll}
\Lambda^{2}\fr{m}_{1}^{*}\wedge\fr{m}_2^{*}=\Span_{\bb{R}}\{e^{123}, e^{124}\}\,, & \fr{m}_1^{*}\wedge\Lambda^{2}\fr{m}_{2}^{*}=\Span_{\bb{R}}\{e^{134}, e^{234}\}\,,\\
\Lambda^{2}\fr{m}_{1}^{*}\wedge\fr{m}_3^{*}=\Span_{\bb{R}}\{e^{125}, e^{126}\}\,,& \fr{m}_1^{*}\wedge\Lambda^{2}\fr{m}_{3}^{*}=\Span_{\bb{R}}\{e^{156}, e^{256}\}\,,\\
\Lambda^{2}\fr{m}_{2}^{*}\wedge\fr{m}_3^{*}=\Span_{\bb{R}}\{e^{345}, e^{346}\}\,,&
\fr{m}_2^{*}\wedge\Lambda^{2}\fr{m}_{3}^{*}=\Span_{\bb{R}}\{e^{356}, e^{456}\}\,.
\end{array}
\]
Computing the action of $\chi_* (e_9)$ on the $3$-forms appearing above,
we see that for non-zero $k, \ell, m,$ none of these  3-forms, or any linear combination of them, belong to the kernel of the isotropy action of $H$.
Let us finally prove that  the condition  $k-\ell-m=0$ is equivalent to
\[
(\fr{m}_1^{*}\wedge\fr{m}_2^{*}\wedge\fr{m}_3^{*})^{H}=\Span_{\bb{R}}\{\upalpha_1, \upalpha_2\}.
\]
This module is independent of the rotation that one may apply to $\fr{m}_{0}$. Moreover,
$
\fr{m}_1^{*}\wedge\fr{m}_2^{*}\wedge\fr{m}_3^{*}=\Span_{\bb{R}}\{e^{135}, e^{145}, e^{136}, e^{146}, e^{235}, e^{245}, e^{236}, e^{246}\}
$
and we obtain
\[
\chi_{*}(e_9)(\upalpha_1)=(k-\ell-m)\upalpha_2, \quad \chi_{*}(e_9)(\upalpha_2)=-(k-\ell-m)\upalpha_1\,,
\]
which shows that $\upalpha_1, \upalpha_2\in(\Lambda^{3}\fr{m}^{*})^{H}$ if and only if  $k-\ell-m=0.$
Finally, it is easy to see that these invariant 3-forms constitute a basis of $(\Lambda^{3}\fr{m}^{*})^{H}$ if and only if $k=\ell+m$.
\end{proof}

Combining  Lemmas \ref{2forms} and \ref{3forms} allows us to describe the invariant 4-forms on  $C_{\ell+m, \ell, m}$.
\begin{lemma}\label{4forms}
Assume that the integers $k, \ell, m,$ satisfy $k-\ell-m=0$, with $k>\ell>m>0$.
Then, the space of $G$-invariant 4-forms  on  $C_{\ell+m, \ell, m}$ is 10-dimensional, and it is generated by the 4-forms
\[
e^{1234},~e^{1256},~e^{1278},~e^{5678},~e^{3478},~e^{3456},
\]
and
\[
\upbeta_1\coloneqq\upalpha_1\wedge e^{7}, \quad \upbeta_2 \coloneqq \upalpha_2\wedge e^{8}, \quad \upzeta_1\coloneqq \upalpha_2\wedge e^{7},  \quad \upzeta_2\coloneqq-\upalpha_1\wedge e^{8}
\]
where $\upalpha_1$ and $\upalpha_2$ are the invariant 3-forms described in (\ref{3forms2}).
Note that $\upbeta_2$ and $\upzeta_2$ are obtained by applying the Hodge star operator of the bi-invariant metric $\langle \cdot , \cdot \rangle$ to $\upbeta_1$ and $\upzeta_1$, respectively.
\end{lemma}

\begin{proof}  By the splitting $\fr{m}=\fr{p}\oplus\fr{m}_0$, we obtain the $\Ad(H)$-invariant orthogonal decomposition
\begin{equation}\label{inv44}
(\Lambda^{4}\fr{m}^{*})^{H} = \Big(\Lambda^{4}(\fr{p}\oplus\fr{m}_0)^{*}\Big)^{H}=(\Lambda^{4}\fr{p}^{*})^{H}\oplus(\Lambda^{3}\fr{p}^{*}\wedge\fr{m}_{0}^{*})^{H}
\oplus(\Lambda^{2}\fr{p}^{*}\wedge\Lambda^{2}\fr{m}_{0}^{*})^{H}.
\end{equation}
The module $(\Lambda^{2}\fr{p}^{*}\wedge\Lambda^{2}\fr{m}_{0}^{*})^{H}$  is generated by the invariant  $4$-forms  $e^{1278}, e^{3478}$ and  $e^{5678}$,
which are obtained by wedging the 3-forms given in (\ref{3forms1}) with the generators of $\fr{m}_{0}$.
Consider now the second summand $(\Lambda^{3}\fr{p}^{*}\wedge\fr{m}_{0}^{*})^{H}$.
We can argue in a similar way as we did for $(\Lambda^{3}\fr{p}^{*})^{H}$.
In detail, we know that $\Lambda^{3}\fr{p}^{*}\wedge\fr{m}_{0}^{*}$ splits into a direct sum of seven subspaces, and it is easy to see that
any isomorphism $\fr{m}_{0}^{\theta}\cong\fr{m}_{0}$ does not contribute with further summands.
An inspection of each subspace allows us to conclude that
only one of them contains elements belonging to the kernel of the isotropy action when $k=\ell+m$, namely $\fr{m}_{1}^{*}\wedge\fr{m}_{2}^{*}\wedge\fr{m}_{3}^{*}\wedge\fr{m}_{0}^{*}$.
More precisely, we have
\[
(\fr{m}_{1}^{*}\wedge\fr{m}_{2}^{*}\wedge\fr{m}_{3}^{*}\wedge\fr{m}_{0}^{*})^{H} = (\fr{m}_{1}^{*}\wedge\fr{m}_{2}^{*}\wedge\fr{m}_{3}^{*})^{H}\wedge(\fr{m}_{0}^{*})^{H}
=\Span_{\bb{R}}\{\upalpha_1\wedge e^{7}, \upalpha_1\wedge e^{8}, \upalpha_2\wedge e^{7}, \upalpha_2\wedge e^{8}\}.
\]
This immediately follows from the identities
\[
\renewcommand\arraystretch{1.2}
\begin{array}{ll}
\chi_{*}(e_9)\upbeta_1=(k-\ell-m)\upzeta_1,	& \chi_{*}(e_{9})\upbeta_2=(k-\ell-m)\upzeta_2, \\
\chi_{*}(e_{9})\upzeta_1=-(k-\ell-m)\upbeta_1,	& \chi_{*}(e_{9})\upzeta_2=-(k-\ell-m)\upbeta_2.
\end{array}
\]
To conclude the proof, we have to examine the first module $(\Lambda^{4}\fr{p}^{*})^{H}$ in \eqref{inv44}. It decomposes as follows:
\begin{eqnarray*}
(\Lambda^{4}\fr{p}^{*})^{H}&=&(\Lambda^{2}\fr{m}_{1}^{*}\wedge\Lambda^{2}\fr{m}_{2}^{*})^{H}\oplus(\Lambda^{2}\fr{m}_{1}^{*}\wedge\Lambda^{2}\fr{m}_{3}^{*})^{H}
\oplus(\Lambda^{2}\fr{m}_{2}^{*}\wedge\Lambda^{2}\fr{m}_{3}^{*})^{H} \\
&&\oplus(\Lambda^{2}\fr{m}_{1}^{*}\wedge\fr{m}_2^{*}\wedge\fr{m}_3^{*})^{H}
\oplus(\fr{m}_1^{*}\wedge\Lambda^{2}\fr{m}_{2}^{*}\wedge\fr{m}_3^{*})^{H}\oplus(\fr{m}_1^{*}\wedge\fr{m}_{2}^{*}\wedge\Lambda^{2}\fr{m}_{3}^{*})^{H}\,,
\end{eqnarray*}
where the first three modules are 1-dimensional and they are generated by $e^{1234}$, $e^{1256}$, $e^{3456}$, respectively.
A direct computation shows that the remaining modules are trivial.
\end{proof}

From the above proposition, we obtain the following.
\begin{corol}\label{4frminv}
Consider the homogeneous  spaces $C_{k, \ell, m}$ with $k>\ell>m>0$. Then, the 4-form
\begin{eqnarray}
\Phi &=&e^{1234}+e^{1256}-e^{1278}+e^{1357}+e^{1368}+e^{1458}-e^{1467}\nonumber\\
&&+e^{5678}+e^{3478}-e^{3456}+e^{2468}+e^{2457}-e^{2358}+e^{2367}\label{invmm}
\end{eqnarray}
is invariant if and only if $k-\ell-m=0$. Whenever this condition is satisfied,  $\Phi$ induces an invariant $\Spin(7)$-structure on $C_{\ell+m, \ell, m}$.
\end{corol}

\begin{remark}\label{NoLCPCklm}
Due to the results of  Appendix \ref{app-detailsCklm}, it is straightforward to check that there are no invariant closed 1-forms on $C_{k, \ell, m}$ as long as the integers $k, \ell, m$ satisfy $k>\ell>m>0$.
Thus, this homogeneous space cannot admit any invariant l.c.p.~ $\Spin(7)$-structure.
\end{remark}

We now describe a 5-parameter family of invariant $\Spin(7)$-structures on $C_{\ell+m,\ell,m}$  inducing the most general invariant metric $( \ , \ )_{y_1,y_2,y_3,y_4, y_5} \coloneqq ( \ , \ )_{y_1,y_2,y_3,y_4, y_5, 0}$ adapted to
the reductive decomposition $\fr{m}=\fr{m}_1\oplus\fr{m}_2\oplus\fr{m}_2\oplus\fr{m}_0$, with $\fr{m}_{0}=\fr{m}_4\oplus\fr{m}_5$ (cf.~Proposition \ref{nik}).
The choice $\theta=0$ for the summand $\fr{m}_{0}^{\theta}$ is just a matter of convenience, as it simplifies the computations afterwards.

\begin{prop}\label{Leeform}
The 4-form
\begin{eqnarray*}
\Phi_{y_1,y_2,y_3,y_4,y_5}&\coloneqq&y_{1}^{2}y_{2}^{2}e^{1234}+y_{1}^{2}y_{3}^{2}e^{1256}-y_{2}^{2}y_{3}^{2}e^{3456}-y_{1}^{2}y_{4}y_{5}e^{1278}+y_{2}^{2}y_{4}y_{5}e^{3478}+y_{3}^{2}y_{4}y_{5}e^{5678}\nonumber\\
						&&+y_{1}y_{2}y_{3}y_4\,\upbeta_1+y_{1}y_{2}y_{3}y_{5}\,\upbeta_2,
\end{eqnarray*}
where $y_1,\ldots,y_5$, are positive real parameters,
defines an invariant $\Spin(7)$-structure  on  $C_{\ell+m, \ell, m}, \ell>m>0$.
It induces the metric $( \ , \ )_{y_1,y_2,y_3,y_4, y_5}$, and its Lee form  is given by
\begin{eqnarray*}
\vartheta_{y_1,y_2,y_3,y_4, y_5}&=&\frac{2}{7\bar{c}_7}\Big(\ell m (y_2^2y_3^2-y_1^2y_3^2+6y_1^2y_2^2)y_{4}^2y_{5}^2+m^2(y_3^2y_4^2+y_1^2y_4^2+4y_1^2y_3^2)y_{2}^2y_{5}^2\\
&&+2\ell^2(y_4^2+2y_3^2)y_1^2y_2^2y_5^2\Big) \,  e^8 - \frac{2(2l+m)}{7\bar{c}_8}\Big((2y_1^2+y_5^2)y_2^2y_3^2y_4^2\Big) \, e^{7},
\end{eqnarray*}
where $\bar{c}_7$ and $\bar{c}_8$ are obtained by setting $k=\ell+m$ in the expressions of $c_7$ and $c_8$, respectively. In particular, the $\Spin(7)$-structure is always of mixed type.
\end{prop}

\begin{proof}
It is clear that the 4-form $\Phi_{y_1,y_2,y_3,y_4, y_5}$ is invariant and admissible and that it induces the metric $( \ , \ )_{y_1, \ldots, y_5}$.
For brevity, we denote by $\star$ the corresponding Hodge operator.
Recall that the Lee form is given by
\[
\vartheta_{y_1,y_2,y_3,y_4, y_5}= -\frac{1}{7}\star(\star{\rm d}\Phi_{y_1,y_2,y_3,y_4, y_5}\wedge \Phi_{y_1,y_2,y_3,y_4, y_5}).
\]
Using the results of Appendix \ref{app-detailsCklm} together with the definition of the Hodge operator, we obtain the expression of $\vartheta_{y_1,y_2,y_3,y_4, y_5}$ described above.
To determine the Fern\'andez type of this $\Spin(7)$-structure, it is sufficient to observe that the Lee form is never closed (cf.~Remark \ref{NoLCPCklm}),
and that the condition $\vartheta_{y_1,y_2,y_3,y_4, y_5}=0$ is equivalent to a system of two polynomial equations in the variables $y_1,\ldots,y_5,\ell,m$, which has no solutions under the constraints $\ell>m>0$ and
$y_i>0$, $i=1,\ldots,5$. Thus, the $\Spin(7)$-structure is of mixed type.
\end{proof}

\begin{corol}\label{CanConCklm}
The  invariant $\Spin(7)$-structure defined by the admissible 4-form $\Phi \coloneqq\Phi_{1, \ldots, 1}$ on $C_{\ell+m, \ell, m}$,
with $\ell > m> 0$,  is of mixed type and its characteristic connection $\nabla$ coincides with the canonical connection $\nabla^{\scriptscriptstyle0}$ with respect to the naturally reductive structure induced by
$g \coloneqq ( \ , \ )_{1, \ldots, 1}$. In particular, its torsion form is parallel,  i.e., $\nabla T=0$.
\end{corol}

\begin{proof}
We already know that the 4-form $\Phi$ defines a $\Spin(7)$-structure of mixed type and it induces the normal metric $g\coloneqq( \ , \  )_{1, \ldots, 1}$.
The homogeneous space $(C_{k, \ell, m}, g)$ is naturally reductive. Consequently, the canonical connection $\nabla^{\scriptscriptstyle0}$ has totally skew-symmetric torsion (cf.~\cite{TrVa}).
Moreover, since $\Phi$ and $g$ are $G$-invariant, they are both parallel with respect to $\nabla^{\scriptscriptstyle0}$ (see e.g.~\cite[Ch.~X, Prop.~2.7]{KoNo}).
Thus, the connection $\nabla^{\scriptscriptstyle0}$ must coincide with the canonical connection $\nabla$ by the uniqueness of the latter \cite[Thm.~1.1]{Iv},
and the $G$-invariant torsion form $T$ is $\nabla$-parallel.

Using the results of Appendix \ref{app-detailsCklm}, we can easily compute the expression of the torsion form, obtaining
\[
T = \star\dd{\Phi}-\frac{7}{6}\star(\vartheta\wedge{\Phi}) = 	-\frac{m(\ell+m)}{\bar{c}_7}\, e^{127}-\frac{\ell}{{\bar{c}_8}}\, e^{128}-\frac{\ell m}{\bar{c}_7}\, e^{347}
											+\frac{\ell+m}{\bar{c}_8}\, e^{348}+\frac{(\ell+m)^{2}+\ell^{2}}{\bar{c}_7}\,e^{567},
\]
where $\bar{c}_7$ and $\bar{c}_8$ are obtained by setting $k=\ell+m$ in the expressions of $c_7$ and $c_8$, respectively.
\end{proof}

To study the case $k=\ell=1$, $m=0$ (cf.~Example \ref{su2su2s2}), we can proceed in a similar way. It turns out that the space of invariant 4-forms on $C_{1,1,0}$ has dimension 26 and that
an invariant $\Spin(7)$-structure on $C_{1,1,0}$ is given by
\begin{eqnarray*}
\Phi		&\coloneqq&e^{1234}+e^{1256}-e^{1278}+e^{1357}+e^{1368}+e^{1458}-e^{1467}\nonumber\\
		&&+e^{5678}+e^{3478}-e^{3456}+e^{2468}+e^{2457}-e^{2358}+e^{2367}.
\end{eqnarray*}
This $\Spin(7)$-structure is of mixed type with associated invariant metric $g_\Phi = \sum_{i=1}^8 (e^i)^2$ and Lee form $\vartheta = -\frac{6\sqrt{2}}{7}\,e_7+\frac67\,e_8$.

\subsection{The Calabi-Eckmann manifold $(\SU(3)/\SU(2))\times\SU(2)$}\label{CEck}
Let $G\coloneqq \SU(3)\times\SU(2)$ and $H\coloneqq \SU(2)$.
$M=G/H$ is a direct product of prime homogeneous spaces,  hence a reductive decomposition of the Lie algebra $\fr{g}=\fr{su}(3)\oplus\fr{su}(2)$ is given by
\[
\fr{g}=\fr{h}\oplus\fr{m}, \quad \fr{m}=\fr{n}\oplus\bb{R}^{3}.
\]
Clearly, $\fr{n}$ coincides with the tangent space of $\SU(3)/\SU(2)$ at the identity coset,  and $\bb{R}^3\cong T_{e}\SU(2)$.
Let us now consider the following basis of $\fr{su}(3)$
\[
\renewcommand\arraystretch{1}
\begin{array}{cccc}
e_1 \coloneqq
 \begin{pmatrix}
0 & -1 & 0 \\
1 &  0 & 0\\
0 & 0 & 0
\end{pmatrix},

&

e_2\coloneqq
\begin{pmatrix}
0 & i & 0 \\
i &  0 & 0\\
0 & 0 & 0
\end{pmatrix},

&

e_3\coloneqq
\begin{pmatrix}
0 & 0 & 1 \\
0 &  0 & 0 \\
-1 & 0 & 0
\end{pmatrix}, 		

&

e_4\coloneqq
\begin{pmatrix}
0 & 0 & i \\
0 &  0 & 0\\
i & 0 & 0
\end{pmatrix}, \\ \\

e_5 \coloneqq \frac{1}{\sqrt{3}}
\begin{pmatrix}
2i & 0 & 0 \\
0 &  -i & 0\\
0 & 0 & -i
\end{pmatrix},

&

u_6\coloneqq
\begin{pmatrix}
0 & 0 & 0 \\
0 &  i & 0\\
0 & 0 & -i
\end{pmatrix},

&

u_7\coloneqq
\begin{pmatrix}
0 & 0 & 0 \\
0 &  0 & 1\\
0 & -1 & 0
\end{pmatrix},

&

u_8\coloneqq
\begin{pmatrix}
0 & 0 & 0 \\
0 &  0 & i\\
0 & i & 0
\end{pmatrix}\,.

\end{array}
\]
With this choice, the stability subalgebra $\fr{h}$ is generated by the triple $\{u_6, u_7, u_8\}$, and $\{e_1, \ldots, e_5\}$ is a basis of $\fr{n} = \V\oplus\bb{R}$,
where the module $\bb{R}$ is spanned by $e_5$ and  $\V\cong \bb{H}=\Span_{\bb{R}}\{e_1, e_2, e_3, e_4\}\cong\Span_{\bb{R}}\{\bold{1}, \bold{i}, \bold{j}, \bold{k}\}\,.$
 Moreover, it is easy to check that this $\fr{su}(3)$-basis is orthonormal with respect to the bi-invariant inner product
\begin{equation}\label{biinvsu}
\langle X, Y\rangle \coloneqq -\frac12\,{\rm tr}\left(XY\right).
\end{equation}

From now on, we shall identify $\V\cong\V^*$ via the quaternionic metric on $\V\cong\bb{H}$.
A basis for  the tangent space $\fr{m}$ is given by the union of $\{e_1, \ldots, e_5\}$ with a basis $\{e_6, e_7, e_8\}$ of $T_{e}\SU(2) \cong \fr{su}(2)$.
The latter may be chosen as follows
\[
e_6\coloneqq -\frac12
\begin{pmatrix}
0 & i  \\
i &  0
\end{pmatrix},
\quad
e_7\coloneqq \frac12
\begin{pmatrix}
0 & -1 \\
1 & 0
\end{pmatrix}, 		
\quad
e_8\coloneqq \frac12
\begin{pmatrix}
-i & 0\\
0 &  i
\end{pmatrix},
\]
so that $[e_6, e_7]=e_8,~[e_6, e_8]=-e_7,~[e_7, e_8]= e_6.$  Since $\fr{m}=\fr{n}\oplus\bb{R}^3$ and the action of $H$ on $\SU(2)\cong\Ss^3$ is trivial, we have
 \[
 [e_{i}, e_{j}]=0=[u_k, e_j], \quad \forall \ i=1, \ldots 5, \ k, j=6, 7, 8. \quad
 \]
Of course,   $e_5$   is also invariant under the isotropy action, i.e.,   $[u_k, e_5]=0$ for any $k=6, 7, 8$, while on $\V$ we have
\[
\begin{array}{llll}
\chi_{*}(u_6)e_1={e_2},	&	\chi_{*}(u_6)e_2= {-e_1},	&	\chi_{*}(u_6)e_3=e_4,				&	\chi_{*}(u_6)e_4=-e_3,\\
\chi_{*}(u_7)e_1= {e_3},	&	\chi_{*}(u_7)e_2=-e_4,				&	\chi_{*}(u_7)e_3= {-e_1}	&	\chi_{*}(u_7)e_4=e_2,\\	
\chi_{*}(u_8)e_1= {e_4},	&	\chi_{*}(u_8)e_2=e_3,				&	\chi_{*}(u_8)e_3=-e_2, 				&	\chi_{*}(u_8)e_4= {-e_1}.
\end{array}
\]
Hence, in terms of skew-symmetric matrices $E_{ij}$ we obtain the orthogonal transformations
 \[
 \chi_{*}(u_6)|_{\fr{m}}=E_{12}+E_{34}, \quad \chi_{*}(u_7)|_{\fr{m}}=E_{13}-E_{24}, \quad \chi_{*}(u_8)|_{\fr{m}}=E_{14}+E_{23}\,.
 \]

We  are now ready to describe the space of $G$-invariant forms on  $M$. As before, we shall denote by $e^{k}\in\fr{m}^{*}$ the dual of $e_{k}$.
 \begin{lemma}\label{calab4forms}  The spaces of $G$-invariant $k$-forms on $M=(\SU(3)/\SU(2))\times\SU(2)$,  for $k=1, 2, 3, 4,$ have the following dimensions:
 \[
\mathrm{dim}\left((\Lambda^1\fr{m}^*)^{H}\right) = 4, \quad \mathrm{dim}\left((\Lambda^2\fr{m}^*)^{H}\right) = 9, \quad \mathrm{dim}\left((\Lambda^3\fr{m}^*)^{H}\right) = 16, \quad
\mathrm{dim}\left((\Lambda^4\fr{m}^*)^{H}\right) = 20.
\]
A basis of each space is described in the proof.
\end{lemma}
\begin{proof}
From the above description of the isotropy action, we immediately see that $(\Lambda^1\fr{m}^*)^{H}$ is spanned by $\{e^5, e^6, e^7, e^8\}$.

Let us consider the space of invariant 2-forms. Since $\fr{m}=\fr{n}\oplus\bb{R}^3$ is an orthogonal $\Ad(H)$-invariant decomposition, we obtain the invariant splitting
\[
(\Lambda^2\fr{m}^*)^{H}=(\Lambda^{2}\fr{n}^*)^{H}\oplus(\Lambda^{1}\fr{n}^*\wedge\bb{R}^{3})\oplus(\Lambda^{2}\bb{R}^{3})^{H}\,,
\]
where we identify $\bb{R}^{3}\cong(\bb{R}^{3})^{*}$. Obviously,  the second module is spanned by $\{e^{56}, e^{57}, e^{58}\}$, and the third one by $\{e^{67}, e^{68}, e^{78}\}$. Thus we are  left with
\[
(\Lambda^{2}\fr{n}^*)^{H}=(\Lambda^{2}\V)^{H}\oplus(\V\wedge\bb{R})^{H}\,.
\]
The summand $(\V\wedge\bb{R})^{H}$ is easily seen to be trivial,
while for the first summand we recall that
\[
\Lambda^{2}\V=\Lambda^{2}\bb{H}\cong\fr{so}(4)=\fr{sp}(1)^{+}\oplus\fr{sp}(1)^{-}=\Span_{\bb{R}}\{e^{12}, e^{13}, e^{14}\}\oplus\Span_{\bb{R}}\{e^{23}, e^{24}, e^{34}\}\,.
\]
Then, a direct computation shows that a basis of $(\Lambda^{2}\V)^{H}$ is given by
$\left\{e^{12}-e^{34},~e^{13}+e^{24},~e^{14}-e^{23}\right\}$.

 Let us now consider the space of invariant 3-forms
\[
 (\Lambda^3\fr{m}^*)^{H}= (\Lambda^3\fr{n}^*)^{H}\oplus(\Lambda^2\fr{n}^*\wedge\bb{R}^3)^{H}\oplus(\Lambda^{1}\fr{n}^{*}\wedge\Lambda^{2}\bb{R}^3)^{H}\oplus(\Lambda^{3}\bb{R}^{3})^{H}\,.
\]
The last two summands are generated by $\left\{e^{567}, e^{568}, e^{578}\right\}$ and $\left\{e^{678}\right\}$, respectively.
The second summand
is spanned by $\left\{e^{12k}-e^{34k},~e^{13k}+e^{24k},~e^{14k}-e^{23k}~:~k=6,7,8\right\}$. Thus, it has dimension 9.
Finally, we have
\[
 (\Lambda^3\fr{n}^*)^{H}= (\Lambda^3\V)^{H}\oplus (\Lambda^2\V\wedge\bb{R})^{H}\,,
\]
where the first summand is trivial, while a basis of invariant 3-forms for the second summand is given by
$\left\{e^{125}-e^{345},~e^{135}+e^{245},~e^{145}-e^{235}\right\}$.
Summing up, the space of invariant 3-forms on $M^8$ is 16-dimensional.

As for the $H$-module $(\Lambda^4\fr{m}^*)^{H}$, we have
\[
(\Lambda^4\fr{m}^*)^{H}=(\Lambda^4\fr{n}^*)^{H}\oplus(\Lambda^3\fr{n}^*\wedge\bb{R}^3)^{H}\oplus(\Lambda^2\fr{n}^*\wedge\Lambda^2\bb{R}^3)^{H}\oplus(\Lambda^1\fr{n}^*\wedge\Lambda^{3}\bb{R}^3)^{H}\,.
\]
The second and the third module are both 9-dimensional. A basis of $(\Lambda^3\fr{n}^*\wedge\bb{R}^3)^{H}\cong (\Lambda^3\fr{n}^*)^{H}\wedge(\bb{R}^3)^{H}$ is given by
\[
\left\{e^{125k}-e^{345k},~e^{135k}+e^{245k},~e^{145k}-e^{235k}~:~k=6,7,8 \right\}
\]
and a basis of $(\Lambda^2\fr{n}^*\wedge\Lambda^2\bb{R}^3)^{H}\cong (\Lambda^2\fr{n}^*)^{H}\wedge(\Lambda^2\bb{R}^3)^{H}$ is
\[
\left\{(e^{12}-e^{34})\wedge\alpha,~(e^{13}+e^{24})\wedge\alpha,~(e^{14}-e^{23})\wedge\alpha~:~\alpha=e^{67}, e^{68}, e^{78}\right\}.
\]
The fourth module is 1-dimensional with generator $\left\{e^{5678}\right\}$.
Also,  $(\Lambda^4\fr{n}^*)^{H}  \cong (\Lambda^4\V)^{H}\oplus(\Lambda^3\V\wedge\bb{R})^{H}$,
where $\left\{e^{1234}\right\}$ spans the first summand while the second one is trivial. This finishes the proof.
 \end{proof}

Using the basis of invariant 4-forms, we obtain the following.
\begin{corol}
The  4-form
\begin{eqnarray*}
\Phi		&\coloneqq&e^{1234}+e^{1256}-e^{1278}+e^{1357}+e^{1368}+e^{1458}-e^{1467}\nonumber\\
		&&+e^{5678}+e^{3478}-e^{3456}+e^{2468}+e^{2457}-e^{2358}+e^{2367}
\end{eqnarray*}
induces a $G$-invariant $\Spin(7)$-structure on $(\SU(3)/\SU(2))\times\SU(2)$.
\end{corol}

\begin{remark}
Using the Koszul formula, it is straightforward to check that there are no $G$-invariant closed 1-forms on $M=G/H$.
Thus, we deduce that the Calabi-Eckmann manifold $(\SU(3)/\SU(2))\times\SU(2)$ cannot admit any invariant l.c.p.~ $\Spin(7)$-structure.
\end{remark}

Now, we  shall construct  a family of invariant  $\Spin(7)$-structures  inducing  the  general invariant metric on $M=G/H$.
To this aim,  we  first  study the space of $G$-invariant Riemannian metrics on $M$, or equivalently,  the space of $\mathrm{Ad}(H)$-invariant inner products on the reductive complement $\fr{m}$.
Recall that the decomposition of $\fr{m}$ into inequivalent irreducible $H$-modules is given by  $\fr{m}=\V\oplus \bb{R}\oplus\bb{R}^3$,
with $H$ acting trivially on $\bb{R}^3\cong\fr{su}(2)$. Consequently, an $\mathrm{Ad}(H)$-invariant inner product $g$ on $\fr{m}$ must be of the form
\[
g = t_1^2\left.\langle\cdot,\cdot\rangle\right|_{\V} + t_2^2\left.\langle\cdot,\cdot\rangle\right|_{\bb{R}} + h,
\]
where $\langle\cdot,\cdot\rangle$ is the inner product on $\fr{su}(3)$ given in \eqref{biinvsu}, $t_1,t_2\in\bb{R}^{\scriptscriptstyle+}$, and $h$ is an inner product on $\bb{R}^3\cong\fr{su}(2)$
corresponding to a left-invariant Riemannian metric on $\SU(2)$.
It follows from \cite{Mil} that there exist suitable real parameters $0<t_3\leq t_4 \leq t_5$ for which
$(\SU(2),h)$ is isometrically isomorphic to $\SU(2)$, endowed with the left-invariant Riemannian metric induced by the following inner product on $\fr{su}(2)$
\[
h_{t_3,t_4,t_5} \coloneqq t_3^2\,(e^6\otimes e^{6})+t_4^2\,(e^7\otimes e^7)+t_5^2\,(e^8\otimes e^8).
\]
In particular, the canonical bi-invariant metric on $\SU(2)$ corresponds to $h_{1,1,1}$.
The general invariant metric on $M$ is then given by the inner product
\[
g_{t_1,t_2,t_3,t_4,t_5} \coloneqq t_1^2\left.\langle\cdot,\cdot\rangle\right|_{V^4} + t_2^2\left.\langle\cdot,\cdot\rangle\right|_{\bb{R}} + h_{t_3,t_4,t_5}.
\]
In a similar way as we did  for $C_{k, \ell, m}$, it is now possible to show  the following.

\begin{prop}
On the {\it Calabi-Eckmann manifold} $(\SU(3)/\SU(2))\times\SU(2)$,  the invariant $\Spin(7)$-structure  given by
\begin{eqnarray*}
\Phi_{t_1,t_2,t_3,t_4,t_5} 		&\coloneqq&	t_1^4\,e^{1234}+t_1^2t_2t_3\,(e^{1256}-e^{3456}) +t_1^2t_4t_5(e^{3478}-e^{1278})+t_1^2t_2t_4(e^{1357}+e^{2457})\\
						&		&	+t_1^2t_3t_5(e^{1368}+e^{2468})+t_1^2t_2t_5(e^{1458}-e^{2358})+t_1^2t_3t_4(e^{2367}-e^{1467})\\
						&		&	+t_2t_3t_4t_5e^{5678},
\end{eqnarray*} 	
 induces  the general invariant metric $g_{t_1, \ldots, t_5},$   and it   is of mixed type with Lee form
 \[
\vartheta_{t_1,t_2,t_3,t_4,t_5} = -\frac17\left(\frac{2\,t_2\left(t_3^2+t_4^2+t_5^2\right)}{t_3t_4t_5}\, e^5  + \frac{4\sqrt{3}\,t_3\left(2t_1^2+t_2^2\right)}{t_1^2t_2}\, e^6 \right).
\]
\end{prop}

Arguing as in the proof of Corollary \ref{CanConCklm}, the next result immediately follows.
\begin{corol}
Let $\Phi$ denote the invariant $\Spin(7)$-structure on $G/H=(\SU(3)/\SU(2))\times\SU(2)$ obtained by setting $t_i=1$, for all $i=1,\ldots,5$.
Then, the homogeneous space $G/H$ endowed with the metric induced by $\Phi$ is naturally reductive, and its canonical connection $\nabla^{\scriptscriptstyle0}$ coincides with the characteristic connection
$\nabla$ of $\Phi$. In particular, the torsion of $\nabla$ is given by
\[
T = -\sqrt{3}\,(e^{125} -e^{345})+e^{678}
\]
and it is $\nabla$-parallel.
\end{corol}

\begin{remark}
All examples of invariant $\Spin(7)$-structures discussed in this section are of mixed type.
Moreover, we observed that none of the homogeneous spaces considered here admits invariant l.c.p.~$\Spin(7)$-structures.
It is an interesting open question to see whether there exist invariant {\it balanced} $\Spin(7)$-structures on these homogeneous spaces.
\end{remark}

\appendix

\section{Presentations for $\Ss^3\times\Ss^3\times\Ss^2$ not included in the family $C_{k, \ell, m}$.}\label{last}
The $8$-manifold $\Ss^3\times\Ss^3\times\Ss^2$ has several different  presentations as homogeneous space.
For example, in Section \ref{cklm} we analysed the pair $(\fr{g}=\fr{su}(2)\oplus\fr{su}(2)\oplus\fr{su}(2), \fr{h}=\fr{u}(1)_{k, \ell, m})$, which induces the family $C_{k, \ell, m}$.
As a manifold, $C_{k, \ell, m}$ is diffeomorphic to $\Ss^{3}\times\Ss^{3}\times\Ss^{2}$ and any homogeneous space of the form $(\SU(2)\times\SU(2)\times\SU(2))/\U(1)$ is equivariantly diffeomorphic to $C_{k, \ell, m}$.
We also discussed some special cases in Examples \ref{su2su2s2} and \ref{C100}.

In this appendix, our goal is to focus on the presentations of $\Ss^3\times\Ss^3\times\Ss^2$ which are not included in the family $C_{k, \ell, m}$ and which are still (almost) effective and simply connected.

Since the symmetric spaces
$(\SO(4)/\SO(3))\times(\SO(4)/\SO(3))\times(\SO(3)/\SO(2))$ and  $(\SU(2)\times\SU(2)/\Delta\SU(2))\times (\SU(2)\times\SU(2)/\Delta\SU(2))\times (\SU(2)/\U(1))$
coincide, the isometry group of the symmetric Riemannian product $\Ss^3\times\Ss^3\times\Ss^2$
is 15-dimensional. Consequently, we can focus on compact Lie algebras $\fr{g}$ with smaller dimension.
Moreover, as we are interested in invariant $\Spin(7)$-structures, we restrict our attention to the case $\rk\fr{h}\leq 3=\rk\fr{spin}(7)$.
In Table \ref{Table2}, we list all non-symmetric pairs $(\fr{g}, \fr{h})$ satisfying the above constraints and which are different from $(\fr{su}(2)\oplus\fr{su}(2)\oplus\fr{su}(2), \fr{u}(1)_{k, \ell, m})$.

\begin{table}[h]
\centering
\renewcommand\arraystretch{1.5}
\begin{tabular}{| c | c | c | c | c | c | }
\hline
& $\dim\fr{g}$ & $\fr{g}$ & $\fr{h}$ & $\rk\fr{h}$ & $\dim\fr{h}$ \\
\hline
(i) 	&  10 &  $\fr{su}(2)\oplus\fr{su}(2)\oplus\fr{su}(2)\oplus\fr{u}(1)$ & $\fr{u}(1)\oplus\fr{u}(1)$ & 2  & 2\\
(ii) 	&  11 &  $\fr{su}(2)\oplus\fr{su}(2)\oplus\fr{su}(2)\oplus\fr{u}(1)\oplus\fr{u}(1)$ & $\fr{u}(1)\oplus\fr{u}(1)\oplus\fr{u}(1)$ & 3 & 3 \\
(iii) 	&  12 &  $\fr{su}(2)\oplus \fr{su}(2)\oplus\fr{su}(2)\oplus\fr{su}(2)$ & $\Delta\fr{su}(2)\oplus\fr{u}(1)$ & 2 & 4 \\
(iv) 	&  13 &  $\fr{su}(2)\oplus \fr{su}(2)\oplus\fr{su}(2)\oplus\fr{su}(2)\oplus\fr{u}(1)$ & $\fr{su}(2)\oplus\fr{u}(1)\oplus\fr{u}(1)$ & 3 & 5 \\
\hline
\end{tabular}
\vspace{0.1cm}
\caption{Pairs $(\fr{g}, \fr{h})$ inducing non-symmetric homogeneous spaces covered by $\Ss^{3}\times\Ss^{3}\times\Ss^{2}$.}\label{Table2}
\end{table}

\begin{remark}
In addition to the pairs considered in Table \ref{Table2}, one may also consider a fifth case with $\dim\fr{g}=14$, i.e., $\fr{g}=\fr{su}(2)\oplus \fr{su}(2)\oplus\fr{su}(2)\oplus\fr{su}(2)\oplus\fr{u}(1)\oplus\fr{u}(1)$.
However, this corresponds to the stability algebra  $\fr{h}=\fr{su}(2)\oplus\fr{u}(1)\oplus\fr{u}(1)\oplus\fr{u}(1)$, which has rank bigger than 3.
\end{remark}

To obtain the list appearing in Table \ref{Table2}, one has to consider pairs $(\fr{g},\fr{h})$ of the form
\begin{eqnarray*}
\fr{g}	&=&	\fr{su}(2)^{\oplus{a}} \oplus \fr{u}(1)^{\oplus{p}} \oplus\fr{su}(2) \cong \fr{su}(2)^{\oplus({a+1})} \oplus \fr{u}(1)^{\oplus{p}},\\
\fr{h}	&=&	\fr{su}(2)^{\oplus{b}} \oplus \fr{u}(1)^{\oplus{q}} \oplus \fr{u}(1) \cong \fr{su}(2)^{\oplus{b}} \oplus \fr{u}(1)^{\oplus(q+1)},
\end{eqnarray*}
with $a=b+2$, $p=q$, and such that the extra factor $\fr{u}(1)$ of $\fr{h}$ sits diagonally inside the extra factor $\fr{su}(2)$ in $\fr{g}$ (this always induces $\Ss^2$).
Although $p=q$, here we use different indices for the summands $\fr{u}(1)^{\oplus{p}}$ and $\fr{u}(1)^{\oplus{q}}$ to emphasize that the abelian factor of $\fr{h}$ does not coincide with the abelian factor of $\fr{g}$.
The pairs appearing in Table \ref{Table2} correspond to the following values of the parameters
\[
\begin{array}{lcclc}
\mbox{Case (i)}		&	a=2,~b=0,~p=q=1,	&	&	\mbox{Case (ii)}	&	a=2,~b=0,~p=q=2,	\\
\mbox{Case (iii)}	&	a=3,~b=1,~p=q=0,	&	&	\mbox{Case (iv)}	&	a=3,~b=1,~p=q=1.
\end{array}
\]

Let us examine Case (i) in detail. Here, the first factor of $\fr{h}$ sits diagonally inside $\fr{t}^{2}\oplus\fr{u}(1)$, where $\fr{t}^{2}$ is a maximal torus of $\fr{su}(2)\oplus\fr{su}(2)$.
In this case, the pair $(\fr{g}, \fr{h})$ is almost effective and it induces the coset $G/H=\frac{\SU(2)\times\SU(2)\times\U(1)}{\U(1)}\times\frac{\SU(2)}{\U(1)}$, which is simply connected.
To see this, it is sufficient to prove that the space $M_1\coloneqq(\SU(2)\times\SU(2)\times\U(1))/\U(1)$ is simply connected, as $G/H$ is the product of this coset with the simply connected space $M_2\coloneqq\SU(2)/\U(1)$.
The principal circle bundle
\[
\U(1)\stackrel{j}{\longrightarrow}  \SU(2)\times\SU(2)\times\U(1)\stackrel{\pi}{\longrightarrow} M_1 =(\SU(2)\times\SU(2)\times\U(1))/\U(1)
\]
induces the exact sequence
\[
\pi_{1}(\U(1))\stackrel{j_{\sharp}}{\longrightarrow} \pi_{1}(\SU(2)\times\SU(2)\times\U(1))\stackrel{\pi_{\sharp}}{\longrightarrow} \pi_{1}(M_1)\stackrel{\partial}{\longrightarrow}\pi_{0}(\U(1))=\{1\},
\]
which reduces to $\bb{Z}\stackrel{\sim}{\longrightarrow}\bb{Z}\stackrel{\pi_{\sharp}}{\longrightarrow} \pi_{1}(M_1)\stackrel{\partial}{\longrightarrow}1$.
Since $\ker\pi_{\sharp}={\rm Im}j_{\sharp}=\bb{Z}$, the map $\pi_{\sharp}$ must be trivial, whence $\pi_{1}(M_1) = \ker\partial={\rm Im}\pi_{\sharp}=\{1\}$. Thus, $M_1$ is simply connected.

Let $G_1\coloneqq\SU(2)\times\SU(2)\times\U(1)$ and $H_1\coloneqq \U(1)$, so that $M_1=G_1/H_1$.
Consider the universal covering $\tilde{\pi}: \tilde{G}_1\to G_1$ of $G_1$ and $\tilde{H}_1=\tilde{\pi}^{-1}(H_1)$.
Then, we obtain the equivariant diffeomorphism $\tilde{G}_1/\tilde{H}_1\cong G_1/H_1$,
from which it follows that $\tilde{G}_1/\tilde{H}_1=\SU(2)\times\SU(2)$, i.e., $G_1/H_1$ is covered by $\Ss^{3}\times\Ss^{3}=\SU(2)\times\SU(2)$.
Consequently, $G/H$ is covered by  $\Ss^{3}\times\Ss^{3}\times\Ss^{2}$. More generally, we have the following.
\begin{prop}
Any homogeneous manifold $G/H$ induced by a pair $(\fr{g}, \fr{h})$ in Table \ref{Table2}, is covered by $\SU(2)\times\SU(2)\times \Ss^2$, which is diffeomorphic to $\Ss^{3}\times\Ss^{3}\times\Ss^{2}$.
\end{prop}

\section{Details on the infinite family $C_{k, \ell, m}$}\label{app-detailsCklm}
Here, we collect some useful computational details related to the family $C_{k, \ell, m}$, with $k,\ell,m$ coprime integers satisfying $k\geq \ell\geq m\geq 0$ and $k>0$.
The notations used in this appendix are those introduced in Sections \ref{cklm} and \ref{CklmExamples}.

Let $\tilde{e}_9\coloneqq \frac{1}{c_9}\,e_9$, where $c_9\coloneqq\sqrt{(k^{2}+\ell^{2}+m^{2})}$, and let $\{e_1, \ldots, e_8\}$ be the $\langle  \cdot , \cdot \rangle$-orthonormal  basis of $\fr{m}$
given in \eqref{orth1}. The non-zero Lie brackets of the basis vectors are the following
\[
\begin{array}{lll}
[e_1, e_2]=-\displaystyle\frac{km }{c_7}\, e_{7} -\frac{\ell }{c_8}\, e_{8}-\frac{k }{c_9}\, \tilde{e}_{9}, 	& [e_1, e_7]=\displaystyle\frac{km}{c_7}\, e_2,		&  [e_1, e_8]=\displaystyle\frac{\ell}{c_8}\,e_2, \\[8pt]
[e_3, e_4]=-\displaystyle\frac{\ell m}{c_7}\,e_{7}+\frac{k }{c_8}\,e_{8}-\frac{\ell}{c_9}\, \tilde{e}_{9},	& [e_2, e_7]=-\displaystyle\frac{km}{c_7}\,e_1, 			& [e_2, e_8]=-\displaystyle\frac{\ell}{c_8}\,e_1,\\[8pt]
[e_5, e_6]=\displaystyle\frac{(k^{2}+\ell^{2})}{c_7}\,e_{7}-\displaystyle\frac{m }{c_9}\,\tilde{e}_{9}, 	& [e_3, e_7]=\displaystyle\frac{\ell m}{c_7}\,e_4, 		& [e_3, e_8]=-\displaystyle\frac{k}{c_8}\,e_4, \\[8pt]
[e_5, e_7]=-\displaystyle\frac{(k^{2}+\ell^{2})}{c_7}\,e_6, & [e_4, e_7]=-\displaystyle\frac{\ell m}{c_7}\,e_3, & [e_4, e_8]=\displaystyle\frac{k}{c_8}\,e_3, \\[8pt]
[e_6,e_7]=\displaystyle\frac{(k^{2}+\ell^{2})}{c_7}\,e_5, & &\\
\end{array}
\]
where   $c_7\coloneqq\sqrt{(k^{2}+\ell^{2})(k^{2}+\ell^{2}+m^{2})}$, and $c_8\coloneqq\sqrt{(k^{2}+\ell^{2})}$.
Moreover, the brackets $[e_9,e_i]$, for $i=1,\ldots,8,$ are given in \eqref{adjon1}.

The differentials of the dual covectors $\{e^1,\ldots, e^8\}$ of the basis vectors $\{e_1,\ldots,e_8\}$ can be computed using the Koszul formula.
In detail, for all $X,Y\in\fr{m}$ we have
\[
\mathrm{d}e^i(X,Y) = -e^i([X,Y]_{\fr{m}}),
\]
whence we obtain
\[
\begin{array}{l}
{\rm d}e^{1}=\displaystyle \frac{km}{c_7}\,e^{27}+\frac{\ell}{c_8}\,e^{28},	\quad {\rm d}e^{2}=\displaystyle-\frac{km}{c_7}\,e^{17}-\frac{\ell}{c_8}\,e^{18},
																\quad {\rm d}e^{3}=\displaystyle \frac{\ell m}{c_7}\,e^{47}- \frac{k}{c_8}\,e^{48},\\ [8pt]	
{\rm d}e^{4}=\displaystyle-\frac{\ell m}{c_7}\,e^{37}+\frac{k}{c_8}\,e^{38},	\quad {\rm d}e^{5}=\displaystyle-\frac{(k^{2}+\ell^{2})}{c_7}\,e^{67},	
																\quad {\rm d}e^{6}=\displaystyle\frac{(k^{2}+\ell^{2})}{c_7}\,e^{57},\\[8pt]		
{\rm d}e^{7}=\displaystyle\frac{k m}{c_7}\,e^{12}+\frac{\ell m}{c_7}\,e^{34}-\frac{(k^{2}+\ell^{2})}{c_7}\,e^{56}, \quad  {\rm d}e^{8}=\displaystyle\frac{\ell}{c_8}\,e^{12}-\frac{k}{c_8}\,e^{34}.
\end{array}
\]
\noindent Using these expressions together with the properties of the differential operator $\dd$, it is possible to compute the exterior derivative of every invariant differential form on $C_{k, \ell, m}$.

\bigskip

\noindent {\bf Acknowledgements:}  I.C.~is grateful to Anton Galaev, Van L\^{e}, Yurii Nikonorov and Henrik Winther for fruitful discussions related to the topic.
He  was supported by the project no.~19-14466Y of Czech Science Foundation (GA\v{C}R). D.A. was partially  supported by
Grant 18-00496S  of  the Czech Science Foundation.  A.F. and A.R. were partially supported by GNSAGA of INdAM  and by PRIN 2017 "Real and Complex Manifolds: Topology, Geometry and holomorphic dynamics.


\end{document}